\documentclass[a4paper,reqno]{amsart}

\usepackage[utf8]{inputenc}
\usepackage{amssymb}
\usepackage{enumitem}
\usepackage{color}
\usepackage{mathrsfs}
\usepackage{mathtools}
\setlist[enumerate]{label=\emph{(\roman*)}}

\usepackage{hyperref}

\newtheorem{theorem}{Theorem}[section]
\newtheorem{corollary}[theorem]{Corollary}
\newtheorem{lemma}[theorem]{Lemma}
\newtheorem{proposition}[theorem]{Proposition}

\theoremstyle{definition}
\newtheorem{definition}[theorem]{Definition}
\newtheorem{remark}[theorem]{Remark}

\numberwithin{equation}{section}
\newcommand\RR{\mathbb{R}}

\newcommand\ud{\, \mathrm{d}}
\newcommand\e{\varepsilon}

\newcommand\md{\textnormal{Mod}}
\newcommand\vve{\vec \e\,}
\newcommand\vev{\vec \phi\,}
\newcommand\vvep{\vec \e_\perp}
\newcommand\tvvep{\vec {\tilde\e}_\perp}

\newcommand\cv{\cce}
\newcommand\ce{\check{\eta}}
\newcommand\czk{{\check{z}}_{k}}
\newcommand\cbk{{\check{\ell}}_{k}}

\newcommand\cce{\check{\e}}
\newcommand\eun{\mathbf{e}_1}
\newcommand\eN{\mathbf{e}_N}

\newcommand\TD{T_\delta}
\newcommand\TS{T_*}
\newcommand\Tu{t_1}
\newcommand\Td{t_2}
\newcommand\ENE{{H^1\times L^2}}

\newcommand\LL{\mathcal L}
\newcommand\BB{\mathcal G}

\DeclareMathOperator{\spn}{span}
\DeclareMathOperator{\Id}{Id}

\begin{document}

\parindent=0pt

\title[2-solitary waves of nonlinear damped Klein-Gordon equations]
{Description and classification of 2-solitary waves for nonlinear damped Klein-Gordon equations}

\author[R.~Côte]{Raphaël Côte}
\address{IRMA UMR 7501, Université de Strasbourg, CNRS, F-67000 Strasbourg, France}
\email{cote@math.unistra.fr}

\author[Y.~Martel]{Yvan Martel}
\address{CMLS, \'Ecole polytechnique, CNRS, Institut Polytechnique de Paris, F-91128 Palaiseau Cedex, France}
\email{yvan.martel@polytechnique.edu}

\author[X.~Yuan]{Xu Yuan}
\address{CMLS, \'Ecole polytechnique, CNRS, Institut Polytechnique de Paris, F-91128 Palaiseau Cedex, France}
\email{xu.yuan@polytechnique.edu}

\author[L. Zhao]{Lifeng Zhao}
\address{School of Mathematical Sciences,
University of Science and Technology of China, Hefei 230026, Anhui, China}
\email{zhaolf@ustc.edu.cn}

\thanks{R.~C. was partially supported by the French ANR contract MAToS ANR-14-CE25-0009-01.
Y.~M. and X.~Y. thank IRMA, Université de Strasbourg, for its hospitality.
L.~Z. thanks CMLS, \'Ecole Polytechnique, for its hospitality.
L.~Z. was partially supported by the NSFC Grant of China (11771415).}

\subjclass[2010]{35L71 (primary), 35B40, 37K40}

\begin{abstract}
We describe completely $2$-solitary waves related to the ground state of the nonlinear damped Klein-Gordon equation 
\[ \partial_{tt}u+2\alpha\partial_{t}u-\Delta u+u-|u|^{p-1}u=0 \]
on $\RR^N$, for $1\leq N\leq 5$ and energy subcritical exponents $p>2$. The description is twofold.

First, we prove that $2$-solitary waves with same sign do not exist. Second, we construct and classify the full family of $2$-solitary waves in the case of opposite signs.
Close to the sum of two remote solitary waves, it turns out that only the components of the initial data in the unstable direction of each ground state are relevant in the large time asymptotic behavior of the solution.
In particular, we show that $2$-solitary waves have a universal behavior: the distance between the solitary waves is asymptotic to $\log t$ as $t\to \infty$.
This behavior is due to damping of the initial data combined with strong interactions between the solitary waves.
\end{abstract}

\maketitle

\section{Introduction}

\subsection{Setting of the problem}
We consider the nonlinear focusing damped Klein-Gordon equation
\begin{equation}\label{nlkg}
\partial_{tt}u+2\alpha\partial_{t}u-\Delta u+u-f(u)=0 \quad (t,x)\in \RR\times\RR^N,
\end{equation}
where $f(u)= |u|^{p-1}u$, $\alpha>0$, $1\le N\leq 5$, and the exponent $p$ corresponds to the energy sub-critical case, \emph{i.e.}
\begin{equation}\label{on:p}
2<p<\infty \text{ for $N=1,2$ and } 2<p<\frac{N+2}{N-2} \text{ for $N=3,4,5$}.
\end{equation}
The restriction $p>2$ is discussed in Remark~\ref{rmkp}.

This equation also rewrites as a first order system for $\vec u=(u,\partial_t u)=(u,v)$
\begin{equation*}
\left\{\begin{aligned}
\partial_t u & = v\\
\partial_t v & = \Delta u - u + f(u) - 2 \alpha v.
\end{aligned}\right.
\end{equation*}
It follows from \cite[Theorem~2.3]{BRS} that the Cauchy problem for~\eqref{nlkg} is locally well-posed in the energy space: for any initial data
$(u_0,v_0)\in H^1(\RR^N)\times L^2(\RR^N)$, there exists a unique (in some class) maximal solution
$u\in C([0,T_{\max}),H^1(\RR^N))\cap C^1([0,T_{\max}),L^2(\RR^N))$ of \eqref{nlkg}. Moreover, if the maximal time of existence $T_{\max}$ is finite, then
$\lim_{t\uparrow T_{\max}} \|\vec u(t)\|_\ENE=\infty$.

Setting $F(u)= \frac{1}{p+1}|u|^{p+1}$ and
\[
E(\vec u)=\frac 12 \int_{\RR^N} \big\{ |\nabla u|^2 + u^2 + (\partial_t u)^2
- 2 F(u)\big\} \ud x,
\]
for any $H^1\times L^2$ solution $\vec u$ of \eqref{nlkg}, it holds
\begin{equation}\label{energy}
E(\vec u(t_2))-E(\vec u(t_1)) = -2 \alpha \int_{t_1}^{t_2} \|\partial_t u(t)\|_{L^2}^2 \ud t.
\end{equation}

In this paper, we are interested in the dynamics of $2$-solitary waves related to the ground state $Q$, which is the unique positive, radial $H^1$ solution of
\begin{equation}\label{eq:elliptic}
-\Delta Q + Q - f(Q)=0, \quad x\in \RR^N.
\end{equation}
(See~\cite{BL,K}.)
The ground state generates the stationary solution $\vec Q=\big(\begin{smallmatrix} Q\\0\end{smallmatrix}\big)$ of~\eqref{nlkg}.
The function $-\vec Q$ as well as any translate $\vec Q(\cdot-z_0)$ are also solutions of~\eqref{nlkg}.

\bigskip

The question of the existence of multi-solitary waves for~\eqref{nlkg} was first addressed by Feireisl in~\cite[Theorem 1.1]{F98}, under suitable conditions on $N$ and $p$, for an even number of solitary waves with specific geometric and sign configurations.
His construction is based on variational and symmetry arguments to treat the instability direction of the solitary waves.
The goal of the present paper is to fully understand $2$-solitary waves by proving non-existence, existence and classification results using dynamical arguments.

\subsection{Main results} 
First let us introduce a few basic notation. Let $\{\eun,\ldots,\eN\}$ denote the canonical basis of $\RR^N$.
We denote by $\mathcal B_{\RR^N}(\rho)$ (respectively, $\mathcal S_{\RR^N}(\rho)$) the open ball (respectively, the sphere) of $\RR^N$ of center the origin and of radius $\rho>0$, for
the usual norm $|\xi|= (\sum_{j=1}^N \xi_j^2 )^{1/2}$.
We denote by $\mathcal B_\ENE(\rho)$ the ball of $H^1\times L^2$ of center the origin and of radius $\rho>0$
for the norm $\|\big( \begin{smallmatrix} \e\\ \eta \end{smallmatrix}\big)\|_\ENE
=\big(\| \e\|_{H^1}^2+\| \eta\|_{L^2}^2\big)^{1/2}$.
We denote $\langle \cdot, \cdot \rangle$ the $L^2$ scalar product for real valued functions $u_i$ or vector-valued functions $\vec u_i = (u_i,v_i)$ ($i=1,2$):
\[ \langle u_1, u_2 \rangle :=  \int u_1(x) u_2(x) \ud x, \quad \langle \vec u_1, \vec u_2 \rangle :=  \int u_1(x) u_2(x) \ud x +\int v_1(x) v_2(x) \ud x. \]

It is well-known that the operator
\begin{equation*}
\LL = -\Delta+1-p Q^{p-1}
\end{equation*}
appearing after linearization of equation~\eqref{nlkg} around $\vec Q$,
has a unique negative eigenvalue $- \nu_0^2$ ($\nu_0>0$). We denote by $Y$ the corresponding normalized eigenfunction
(see Lemma~\ref{le:L} for details and references).
In particular, it follows from explicit computations that setting
\begin{equation*}
\nu^\pm = - \alpha \pm \sqrt{\alpha^2+\nu_0^2}\quad\text{and}\quad
\vec Y^\pm = \begin{pmatrix}
Y\\ \nu^\pm Y
\end{pmatrix},
\end{equation*}
the function $\vve^{\pm}(t,x) = \exp(\nu^\pm t) \vec Y^\pm(x)$ is solution of the linearized problem
\begin{equation}\label{eq:lin}
\left\{\begin{aligned}
\partial_t \e & = \eta\\
\partial_t \eta & = -\LL \e - 2 \alpha \eta.
\end{aligned}\right.
\end{equation}
Since $\nu^+>0$, the solution $\vve^+$ illustrates the exponential instability
of the solitary wave in positive time.
In particular, we see that the presence of the damping $\alpha>0$ in the equation
does not remove the exponential instability of the Klein-Gordon solitary wave.
An equivalent formulation of instability is obtained by setting
\begin{equation*}
\zeta^\pm = \alpha \pm \sqrt{\alpha^2+\nu_0^2}
\quad\text{and}\quad
\vec Z^\pm = \begin{pmatrix}
\zeta^\pm Y\\ Y
\end{pmatrix}
\end{equation*}
and observing that for any solution $\vve$ of~\eqref{eq:lin},
\begin{equation*}
a^\pm = \langle \vve , \vec Z^\pm\rangle \quad \text{satisfies} \quad
\frac{\ud a^\pm} {\ud t} = \nu^\pm a^\pm.
\end{equation*}

We start with the definition of $2$-solitary waves.
\begin{definition}\label{def:2sol}
A solution $\vec u\in \mathcal C([T,\infty),H^1\times L^2)$ of~\eqref{nlkg}, for some $T\in \RR$, is called a \emph{$2$-solitary wave} if there exist
$\sigma_1, \sigma_2=\pm 1$, a sequence $t_{n}\to \infty$
and a sequence $\left(\xi_{1,n},\xi_{2,n}\right)\in \mathbb{R}^{2N}$ such that
\begin{equation*}
\lim_{n\to\infty}\bigg\{\Big\|u(t_{n})-\sum_{k=1,2}\sigma_{k}Q(\cdot-\xi_{k,n})\Big\|_{H^1}
+\|\partial_{t}u(t_{n})\|_{L^{2}}\bigg\}=0
\end{equation*}
and $\lim_{n\to \infty} \left|\xi_{1,n}-\xi_{2,n}\right|\to \infty$.
\end{definition}
\begin{remark}
We observe that if $u$ is a global solution of~\eqref{nlkg} satisfying
\begin{equation}\label{alter}
\lim_{t\to \infty} \Big\|u(t)-\sum_{k=1,2} \sigma_k Q(\cdot-\xi_k(t))\Big\|_{H^1}=0
\end{equation}
where $\lim_{t\to \infty} |\xi_1(t)-\xi_2(t) |\to \infty$,
then $u$ is a $2$-solitary wave.
Indeed, it follows from~\eqref{energy} and~\eqref{alter} that $t\mapsto E(\vec u(t))$ is lower bounded. Thus, from~\eqref{energy},
it holds $\int_0^\infty \|\partial_t u(t)\|_{L^2}^2 \ud t<\infty$ and so $\lim_{n\to \infty} \|\partial_{t}u(t_{n})\|_{L^{2}} = 0$ for some sequence $t_{n}\to \infty$.
\end{remark}

Our first result concerns the non-existence of $2$-solitary waves with same signs.
\begin{theorem}\label{th:1}
There exists no $2$-solitary wave of~\eqref{nlkg} with $\sigma_1=\sigma_2$.
\end{theorem}
In the next result, we show that $2$-solitary waves with opposite sign satisfy a universal asymptotic behavior.

\begin{theorem}[Description of $2$-solitary waves] \label{th:2}
For any $2$-solitary wave $u$ of~\eqref{nlkg}, there exist $\sigma_1=\pm1$, $\sigma_2=-\sigma_1$, $T>0$ and
$t\in [T,\infty)\mapsto (z_1(t),z_2(t))$ such that for all $t\in [T,\infty)$,
\begin{equation}\label{eq:th:2}
\Big\|u(t)-\sum_{k=1,2}\sigma_{k}Q(\cdot- z_k(t))\Big\|_{H^1}
+\|\partial_{t}u(t)\|_{L^{2}}\lesssim t^{-1},
\end{equation}
and for some constant $c_0=c_0(N)$,
\begin{equation}\label{eqq:th:2}
\lim_{t\to \infty} \left\{|z_1 (t)-z_2 (t)|- \left(\log t-\frac 12 (N-1) \log\log t+c_0\right)\right\}= 0.
\end{equation}
Moreover, there exists $\omega^\infty \in \mathcal{S}_{\RR^N}(1)$ such that
\begin{equation}\label{eqqq:th:2}
\lim_{t\to \infty} \frac{z_1(t)-z_2(t)}{\log t} =2\lim_{t\to \infty} \frac{z_1(t)}{\log t} 
= -2\lim_{t\to \infty} \frac{z_2(t)}{\log t} =\omega^\infty.
\end{equation}
\end{theorem}

Finally, we describe the full family of $2$-solitary waves for initial data close to the sum of two remote solitary waves.

\begin{theorem}[Classification of $2$-solitary waves] \label{th:3}
There exist $C,\delta>0$ and a Lipschitz map
\begin{equation*}
H: (\RR^N\setminus\bar{\mathcal{B}}_{\RR^N}(10|\log \delta|))\times \mathcal B_\ENE(\delta) \to \RR^2, \quad
(L,\vev) \mapsto H(L,\vev) 
\end{equation*}
such that
\begin{equation*}
|H(L,\vev)|< C \left( e^{-\frac L{2}}+\|\vev\|_{H^1\times L^2}\right),
\end{equation*}
with the following property. Given any $L, \vev, h_1,h_2$ such that
\[ |L| > 10|\log \delta|, \quad \| \vev \|_{H^1 \times L^2} < \delta, \quad |h_1| + |h_2| < \delta, \]
%\[ (L,\vev)\in (\RR^N\setminus\bar{\mathcal{B}}_{\RR^N}(10|\log \delta|))\times \mathcal B_\ENE(\delta),\quad
%(h_{1},h_{2})\in \mathcal B_{\RR^2}(\delta), \]
the  solution $\vec u$ of~\eqref{nlkg} with initial data
\begin{equation*}
\vec{u}(0)=
\left(\vec Q+h_1\, \vec{Y}^+\right) \left( \cdot-\frac L2 \right)
-\left(\vec Q+h_2\, \vec{Y}^+\right) \left (\cdot+\frac L2 \right)
+\vev
\end{equation*}
is a $2$-solitary wave if and only if $(h_1,h_2)=H(L,\vev)$.
\end{theorem}

This result essentially means that locally around the sum of two sufficiently separated solitons with opposite signs, the initial data of $2$-solitary waves form a codimension-2 Lipschitz manifold (the unstable directions being directed by $\vec Y^+$ translated around each soliton).

\medskip

We refer to~\S\ref{S:2.4} for a formal discussion on the  dynamics of $2$-solitary waves of~\eqref{nlkg}
justifying the main results of Theorems~\ref{th:1}, \ref{th:2} and~\ref{th:3}.

\begin{remark}\label{rmkp}
We discuss the condition on $p$ in~\eqref{on:p}. The energy sub-criticality condition is necessary for the existence of solitary waves
and allows to work in the framework of finite energy solutions.
The condition $p>2$ could be waived for some of the above results, but it would complicate the analysis and weaken the results. Keeping in mind that the most relevant case is $p=3$, we will not pursue further here the question of lowering $p$.
\end{remark}

\subsection{Previous results}
The question of the long time asymptotic behavior of solutions of the damped Klein-Gordon equation in relation with the bound states
was addressed in several articles; see \emph{e.g.}~\cite{BRS,F94,F98,J,LZ}.
Notably, under some conditions on $N$ and $p$, results in~\cite{F98,LZ} state that for any sequence of time, any global bounded solution of~\eqref{nlkg} converges to a sum of decoupled bound states after extraction of a subsequence of times.
(Note that such result would allow us to weaken the definition of $2$-solitary wave given in Definition~\ref{def:2sol}; however, we have preferred  a  stronger definition valid in any case.)
In~\cite{BRS}, for radial solutions in dimension $N\geq 2$, the convergence of any global solution to one equilibrium is proved to hold for the whole sequence of time. As discussed in~\cite{BRS},
such results are closely related to the general soliton resolution conjecture for global bounded solutions of dispersive problems; see~\cite{DKM,DJKM}
for details and results related to this conjecture for the undamped energy critical wave equation.

The existence and properties of multi-solitary waves is a classical question for integrable models (see for instance~\cite{Miura} for the Korteweg-de Vries equation and \cite{ZS} for the $1$D cubic Schr\"odinger equation).
As mentioned above, \cite{F98} gave the first construction of such solutions for~\eqref{nlkg}.
Since then the same question has been addressed for various non-integrable and undamped nonlinear dispersive and wave equations.
We refer to~\cite{Co,CMMgn,CMkg,MMwave} for the generalized Korteweg-de Vries equation, the nonlinear Schr\"odinger equation, the Klein-Gordon equation and the wave equation, in situations where unstable ground states are involved. See also references therein for previous works related to stable ground states. In those works, the distance between two traveling waves is asymptotic to $Ct$ for $C>0$, as $t\to\infty$.
The more delicate case of multi-solitary waves with logarithmic distance is treated in~\cite{Jkdv,MN,TVNkdv,TVN} for Korteweg-de Vries and Schr\"odinger type equations and systems, both in stable and unstable cases. Note that the logarithmic distance in the latter works is non-generic  while it is the universal behavior for the damped equation \eqref{nlkg}.
See also~\cite{JJnon,JJ,JL} for works on the non-existence,   existence and   classification of radial two-bubble solutions for the energy critical wave equation in large dimensions.

The construction of (center-) stable manifolds in the neighborhood of unstable ground state was addressed in several situations, see \emph{e.g.} \cite{BJ,KS,KNS,MaMeNaRa,NS1}.

While the initial motivation and several technical tools  originate from some of the above mentioned papers, we point out that the present article is self-contained
except for the local Cauchy theory for~\eqref{nlkg} (see \cite{BRS}) and elliptic theory for \eqref{eq:elliptic} and its linearization (we refer to \cite{BL,CMkg,K}).

\bigskip

This paper is organized as follows.
Section~\ref{S:2} introduces all the technical tools involved in a dynamical approach to the $2$-solitary wave problem for~\eqref{nlkg}:
computation of the nonlinear interaction,
modulation, parameter estimates and energy estimates. Theorems~\ref{th:1} and~\ref{th:2} are proved in Section~\ref{S:3}. Finally, 
Theorem~\ref{th:3} is proved in Section~\ref{S:4}.

\subsection{Recollection on the ground state}
The ground state $Q$ rewrites $Q(x)=q\left(|x|\right)$ where $q>0$ satisfies
\begin{equation}
q''+\frac{N-1}{r}q'-q+q^{p}=0,\quad q'(0)=0,\quad \lim_{r\to \infty}q(r)=0.
\end{equation}
It is well-known and easily checked that for a constant $\kappa>0$,
for all $r>1$,
\begin{equation}\label{Qdec}
\left|q(r)-\kappa r^{-\frac{N-1}{2}}e^{-r}\right|+\left|q'(r)+\kappa r^{-\frac{N-1}{2}}e^{-r}\right|
\lesssim r^{-\frac{N+1}{2}}e^{-r}.
\end{equation}
Due to the radial symmetry, there hold the following cancellation (which we will use repetitively):
\begin{equation} \label{eq:Q_sym}
\forall i \ne j, \quad \int \partial_{x_i} Q(x) \partial_{x_j} Q(x) \ud x =0.
\end{equation}
Let
\[
 \LL = -\Delta+1-p Q^{p-1} ,\quad
\langle \LL \e,\e\rangle = \int \big\{|\nabla \e|^2+\e^{2} - p Q^{p-1} \e^2\big\}\ud x.
\]
We recall standard properties of the operator $\LL $ (see \emph{e.g.}~\cite[Lemma 1]{CMkg}).
\begin{lemma}\label{le:L}
\begin{enumerate}
\item \emph{Spectral properties.} The unbounded operator $\LL $ on $L^2$ with domain $H^2$ is self-adjoint, its continuous spectrum is $[1,\infty)$, its kernel is $\spn\{\partial_{x_j}Q : j= 1,\ldots,N\}$ and it has a unique negative eigenvalue $-\nu_{0}^{2}$, with corresponding smooth normalized radial eigenfunction $Y$ $(\|Y\|_{L^2}=1)$
Moreover, on $\RR^N$,
\begin{equation*}
\left|\partial^{\beta}_{x}Y(x)\right|\lesssim e^{-\sqrt{1+\nu_{0}^{2}}\left|x\right|}\quad \text{for any } \beta=\left(\beta_1 ,\ldots,\beta_N\right)\in \mathbb{N}^N.
\end{equation*}
\item \emph{Coercivity property.} There exists $c>0$ such that, for all $\e\in H^{1}$,
\begin{equation*}
\langle \LL \e,\e\rangle\ge c
\|\e\|_{H^{1}}^{2}-c^{-1}
\bigg(\langle \e,Y\rangle^{2} + \sum_{j=1}^N\langle \e,\partial_{x_j}Q\rangle^{2}\bigg).
\end{equation*}
\end{enumerate}
\end{lemma}

\section{Dynamics of two solitary waves}\label{S:2}
We prove in this section a general decomposition result close to the sum of two decoupled solitary waves.
Let any $\sigma_1=\pm 1$, $\sigma_2=\pm1$ and denote $\sigma=\sigma_1\sigma_2$.
Consider time dependent $\mathcal{C}^{1}$ parameters $(z_1 ,z_2 ,\ell_1 ,\ell_2) \in \mathbb{R}^{4N}$
with $\left|\ell_1 \right|\ll1$, $\left|\ell_2 \right|\ll1$ and $\left|z\right|\gg1$ where
\begin{equation*}
 z=z_1 -z_2 \quad\text{and}\quad \ell=\ell_1 -\ell_2.
\end{equation*}
Define the modulated ground state solitary waves, for $k=1,2$,
\begin{equation}\label{def:gs}
Q_{k}=\sigma_{k}Q(\cdot-z_{k})\quad {\rm{and}} \quad
\vec{Q}_{k}=\begin{pmatrix} Q_{k} \\ -({\ell}_{k}\cdot\nabla) Q_{k} \end{pmatrix}.
\end{equation}
Set
\begin{equation*}
R=Q_1 +Q_2 , \quad \quad
\vec{R}=\vec{Q}_1 +\vec{Q}_2 ,
\end{equation*}
and the nonlinear interaction term
\begin{equation}\label{def:G}
G=f(Q_1 +Q_2 )-f(Q_1)-f(Q_2).
\end{equation}
The following functions are related to the exponential instabilities around each solitary wave:
\begin{equation*}
Y_k=\sigma_{k} Y(\cdot-z_{k}),\quad
\vec Y_k^\pm= \sigma_{k} \vec Y^\pm (\cdot-z_{k}),\quad
\vec Z_k^\pm= \sigma_{k} \vec Z^\pm (\cdot-z_{k}).
\end{equation*}

\subsection{Nonlinear interactions}
A key of the understanding of the dynamics of $2$-solitary waves is the computation of the first order of the projections of the nonlinear interaction term $G$ on the directions $\nabla Q_1$ and $\nabla Q_2$ (see \emph{e.g.} \cite[Lemma~7]{TVN}).

\begin{lemma}\label{le:nonlin}
The following estimates hold for $|z|\gg1$.
\begin{enumerate}
\item\emph{Bounds.} For any $0<m'<m$,
\begin{gather}
\int |Q_1 Q_2 |^m \lesssim e^{-m' |z|},\label{tech1}\\
\int\left|F(R)-F(Q_1)-F(Q_2)-f(Q_1)Q_2-f(Q_2)Q_1\right|
 \lesssim e^{-\frac 54|z|}.\label{tech3}
\end{gather}
\item\emph{Sharp bounds.}
For any $m>0$,
\begin{gather}
\int |Q_1| |Q_2|^{1+m} \lesssim q(|z|), \label{tech100}\\
\|G\|_{L^2}\lesssim \|Q_1^{p-1}Q_2\|_{L^2}+ \|Q_1Q_2^{p-1}\|_{L^2}
\lesssim q(|z|).\label{tech2}
\end{gather}
\item\emph{Asymptotics.} It holds
\begin{equation}\label{new3}
\left|\langle f(Q_2),Q_1\rangle - \sigma   c_1 g_0 q(|z|)\right|\leq |z|^{-1} q(|z|)
\end{equation}
where
\begin{equation}\label{on:g0}
g_0= \frac 1{c_1} \int Q^{p}(x)e^{-x_1 }\ud x>0,\quad c_1=\|\partial_{x_1}Q\|_{L^2}^{2}.
\end{equation}
\item\emph{Sharp asymptotics.}
There exists a smooth function $g:[0,\infty)\to \RR$ such that, for any 
$0<\theta<\min(p-1,2)$ and $r>1$
\begin{equation}\label{on:gg}
|g(r)-g_0 q(r)|\lesssim r^{-1} q(r)\end{equation}
and
\begin{align}
\left|\langle G,\nabla Q_1\rangle - \sigma c_1 \frac{z}{|z|}g(|z|)\right|\lesssim e^{-\theta|z|},\label{on:G}\\
\left|\langle G,\nabla Q_2\rangle + \sigma c_1 \frac{z}{|z|}g(|z|)\right|\lesssim e^{-\theta|z|}.\label{on:G2}
\end{align}
\end{enumerate}
\end{lemma}
\begin{proof}
(i) By~\eqref{Qdec}, $|Q(y)|\lesssim e^{-|y|}$, and thus
\begin{equation*}
\int |Q_1 Q_2 |^m \ud y
\lesssim \int e^{-m |y|} e^{-m' |y+z|} \ud y
\lesssim e^{-m' |z|} \int e^{-(m-m')|y|} \ud y
\lesssim e^{-m' |z|}.
\end{equation*}
Then, we observe that
using $p\geq 2$, by Taylor expansion:
\begin{equation*}
\left|F(Q_1+Q_2)-F(Q_1)-F(Q_2)-f(Q_1)Q_2-f(Q_2)Q_1\right|\lesssim |Q_1Q_2|^{\frac 32},
\end{equation*}
which reduces the proof of \eqref{tech3} to applying \eqref{tech1} with $m=\frac 32$
and $m'=\frac 54$.

(ii) We estimate
\begin{align*}
&\int |Q_1| |Q_2|^{1+m} \ud x = \int Q(y-z) Q^{1+m}(y) \ud y\\
&\quad\lesssim q(|z|)\int_{|y|<\frac 34 |z|} e^{|y|} Q^{1+m}(y) \ud y
+e^{-|z|} \int_{|y|>\frac 34 |z|} e^{|y|} Q^{1+m}(y) \ud y\lesssim q(|z|).
\end{align*}
Next, using
\begin{equation*}
|G|\lesssim |Q_1|^{p-1} |Q_2|+|Q_1||Q_2|^{p-1},
\end{equation*}
and (using $p>2$)
\begin{align*}
&\int Q_1^2 |Q_2|^{2(p-1)} \ud x = \int Q^2(y-z) Q^{2(p-1)}(y) \ud y\\
&\quad\lesssim [q(|z|)]^2 \int_{|y|<\frac 34 |z|} e^{2|y|} Q^{2(p-1)}(y) \ud y
+e^{-2|z|} \int_{|y|>\frac 34 |z|} e^{2|y|} Q^{2(p-1)}(y) \ud y
\\&\quad \lesssim [q(|z|)]^2.
\end{align*}

(iii) We claim the following estimate
\begin{equation}\label{on:H}
\left|\int Q^p(y) Q(y+z)\ud y- c_1 g_0 \kappa|z|^{-\frac{N-1}2} e^{-|z|} \right|\lesssim |z|^{-1} q(|z|).
\end{equation}
Observe that~\eqref{new3} follows directly from~\eqref{on:H} and~\eqref{Qdec}.

Proof of~\eqref{on:H}.
First, for $|y|<\frac{3}{4}|z|$ (and so $|y+z|\geq |z|-|y|\geq \frac 14 |z|\gg 1$), we have using~\eqref{Qdec},
\begin{equation*}
\left|Q(y+z)-\kappa|y+z|^{-\frac{N-1}{2}}e^{-|y+z|}\right|
\lesssim \left|y+z\right|^{-\frac{N+1}{2}}e^{-|y+z|}\lesssim \left|z\right|^{-\frac{N+1}{2}}e^{-|z|}e^{|y|}.
\end{equation*}
In particular,
\begin{equation}\label{inter}
\bigg|\int_{|y|<\frac{3}{4}|z|}
Q^{p}(y)\Big[ Q(y+z)- \kappa |y+z|^{-\frac{N-1}{2}}e^{-|y+z|}\Big]\ud y\bigg|
\lesssim\left|z\right|^{-\frac{N+1}{2}}e^{-|z|}.
\end{equation}
Moreover, for $|y|<\frac{3}{4}|z|$, we have the expansions
\begin{align*}
&\left||y+z|^{-\frac{N-1}{2}}-|z|^{-\frac{N-1}{2}}\right|\lesssim |z|^{-\frac{N+1}{2}}|y|,\\
&\left||y+z|-|z|-\frac{y\cdot z}{|z|}\right|\lesssim |z|^{-1}|y|^{2},
\end{align*}
and so
\begin{equation*}
\left||y+z|^{-\frac{N-1}{2}}e^{-|y+z|}-|z|^{-\frac{N-1}{2}}e^{-|z|-\frac{y\cdot z}{|z|}}\right|
\lesssim |z|^{-\frac{N+1}{2}}e^{-|z|} (1+|y|^{2})e^{|y|}.
\end{equation*}
Inserted into~\eqref{inter}, this yields
\begin{equation*}
\left|\int_{|y|<\frac{3}{4}|z|} Q^{p}(y)\Big[Q(y+z)- \kappa |z|^{-\frac{N-1}{2}}e^{-|z|-\frac{y\cdot z}{|z|}}\Big]\ud y\right|
\lesssim\left|z\right|^{-\frac{N+1}{2}}e^{-|z|}.
\end{equation*}

Next, using~\eqref{Qdec}, we observe
\begin{equation*}
\int_{|y|>\frac{3}{4}|z|} Q^{p}(y) Q(y+z)\ud y \lesssim e^{-\frac 34p|z|},
\end{equation*}
and
\begin{equation*}
 \int_{|y|>\frac{3}{4}|z|}Q^{p}(y)e^{-|z|-\frac{y\cdot z}{|z|}}\ud y
\lesssim e^{-|z|} \int_{|y|>\frac{3}{4}|z|}Q^{p}(y) e^{|y|}\ud y\lesssim e^{-\frac 34p|z|}.
\end{equation*}

Gathering these estimates, we have proved
\begin{equation*}
\left|\int Q^p(y) Q(y+z)\ud y - \kappa|z|^{-\frac{N-1}{2}}e^{-|z|} \int Q^p(y) e^{-\frac{y\cdot z}{|z|}} \ud y\right|
\lesssim |z|^{-\frac{N+1}{2}}e^{-|z|}.
\end{equation*}
Last, the identity
$\int Q^{p}(y)e^{-\frac{y\cdot z}{|z|}}\ud y =\int Q^{p}(y)e^{-y_1 }\ud y$
(recall that $Q$ is radially symmetric) and the definition of $g_0$ imply~\eqref{on:H}.

(iv)
First, using the Taylor formula, it holds
\begin{equation*}
|G-p|Q_1|^{p-1}Q_2 |\lesssim |Q_1 |^{p-2}Q_2 ^{2}+|Q_2 |^{p-1}|Q_1 |.
\end{equation*}
Thus, using~\eqref{tech1}, we obtain for any $1<\theta<\min(p-1,2)$,
\begin{equation}\label{22av}
\left|\langle G,\nabla Q_1\rangle
-\sigma_1 \sigma_2  H(z)\right|
\lesssim \int Q^2(y)Q^{p-1}(y+z) \ud y \lesssim e^{-\theta|z|}
\end{equation}
where we set $H(z)= \int \nabla(Q^p)(y) Q(y+z) \ud y$.
Second, we claim that there exists a function $g:[0,\infty)\to \RR$ such that
$H(z)= c_1 \frac{z}{|z|} g(|z|)$.
Indeed, remark by using the change of variable $y=2x_1\eun-x$ that
\begin{align*}
H(r\eun)&=p\int \frac{y}{|y|} q'(|y|) q^{p-1}(|y|)q(|y+r\eun|) \ud y\\
&=p  \int \frac{y_1 \eun}{|y|} q'(|y|) q^{p-1}(|y|)q(|y+r\eun|) \ud y.
\end{align*}
Thus, we set
\begin{equation*}
g(r)=\frac{\eun \cdot H(r\eun)}{c_1}\quad \text{so that}\quad
H(r\eun)=c_1\eun g(r).
\end{equation*}
Let $\omega\in\mathcal{S}_{\RR^N}(1)$ be such that $z=|z|\omega$ and let $U$ be an orthogonal matrix
of size $N$ such that $U\eun=\omega$. Then, using the change of variable $y=Ux$,
\begin{align*}
H(z)&=p\int \frac{y}{|y|} q'(|y|) q^{p-1}(|y|)q(|y+z|) \ud y\\
&=p\int \frac{Ux}{|x|} q'(|x|) q^{p-1}(|x|)q(|x+|z|\eun|) \ud x
=U H(|z|\eun)=c_1 \frac{z}{|z|} g(z).
\end{align*}
Together with~\eqref{22av}, this proves~\eqref{on:G}. The proof of~\eqref{on:G2} is the same but it is important to notice the change of sign due to $H(-z)=-H(z)$.

Last, we observe that proceeding as in the proof of~\eqref{on:H}, it holds
\begin{equation*}
\left|\eun\cdot H(r\eun)
 - \kappa r^{-\frac{N-1}{2}}e^{-r} \int \partial_{x_1}(Q^p)(y) e^{-y_1} \ud y\right|
\lesssim r^{-\frac{N+1}{2}}e^{-r}.
\end{equation*}
Moreover, by integration by parts,
$\int \partial_{x_1}(Q^p)(y)e^{-y_1} \ud y 
= \int Q^p(y) e^{-y_1} \ud y,$
which proves~\eqref{on:gg}.
\end{proof}

\subsection{Decomposition around the sum of two solitary waves}
The following quantity measures the proximity of a function $\vec u =(u,v)$ to the sum of two distant solitary waves,
for $\gamma>0$,
\[
d(\vec u\, ;\gamma)=
\inf_{|\xi_1 -\xi_2| > |\log \gamma|}
\Big \|u-\sum_{k=1,2}\sigma_k Q(\cdot-\xi_k)\Big\|_{H^{1}}
+ \|v\|_{L^2}.
\]
We state a decomposition result for solutions of~\eqref{nlkg}.

\begin{lemma}\label{le:dec}
There exists $\gamma_0>0$ such that for any $0<\gamma<\gamma_0$, $T_1\leq T_2$, and any solution $\vec u=(u,\partial_t u)$ of~\eqref{nlkg} on $[T_1 ,T_2 ]$
satisfying
\begin{equation}\label{for:dec}
\sup_{t\in [T_1 ,T_2 ]} d(\vec u(t);\gamma) <\gamma,
\end{equation}
there exist unique $\mathcal C^1$ functions
\begin{equation*}
t\in [T_1 ,T_2 ]\mapsto ( z_1 ,z_2 ,\ell_1 ,\ell_2)(t)\in
\mathbb R^{4N},
\end{equation*}
such that the solution $\vec{u}$ decomposes on $[T_1 ,T_2 ]$ as
\begin{equation}\label{def:ee}
\vec u = \begin{pmatrix} u \\ \partial_t u\end{pmatrix}=
\vec Q_1+\vec{Q}_2 + \vve,\quad
\vve=\begin{pmatrix}\e \\ \eta \end{pmatrix}
\end{equation}
with the following properties on $[T_1 , T_2 ]$.

\begin{enumerate}
\item \emph{Orthogonality and smallness.} For any $k=1,2$, $j=1,\ldots,N$,
\begin{equation}\label{ortho}
\langle \e, \partial_{x_j} Q_{k}\rangle=\langle \eta,\partial_{x_j} Q_{k}\rangle=0
\end{equation}
and
\begin{equation}\label{eq:bound}
\|\vve\|_\ENE +\sum_{k=1,2} |\ell_k|+e^{-2|z|}\lesssim\gamma.
\end{equation}

\item \emph{Equation of $\vve$.}
\begin{equation}\label{syst_e}\left\{\begin{aligned}
\partial_t \e & = \eta + \md _{{\mathbf \e}}\\
\partial_t \eta &
= \Delta \e-\e+f(R+\e)-f(R)
-2\alpha\eta + \md_{\eta}+G
\end{aligned}\right.\end{equation}
where
\begin{align*}
&\mathrm{Mod}_{\e}=
\sum_{k=1,2}\left(\dot z_{k}-
{\ell}_{k}\right)\cdot\nabla Q_{k},\\
&\mathrm{Mod}_{\eta}=
\sum_{k=1,2}\big(\dot{\ell}_{k}+2\alpha{\ell}_{k}\big)\cdot\nabla Q_{k}-\sum_{k=1,2}({\ell}_{k}\cdot\nabla)(\dot{z}_{k}\cdot\nabla)Q_{k}.
\end{align*}
\item \emph{Equations of the geometric parameters.} For $k=1,2$,
\begin{align}
|\dot{z}_{k}-{\ell}_{k}|&\lesssim \|\vve \|^{2}_\ENE+\sum_{k=1,2}|\ell_{k}|^{2},
\label{eq:z}\\
|\dot{{\ell}}_{k}+2\alpha{\ell}_{k}|
&\lesssim \|\vve\|_\ENE^{2}+\sum_{k=1,2}|\ell_{k}|^{2}+ q(|z|).\label{eq:l}
\end{align}
\item\emph{Refined equation for $\ell_{k}$.}
For any $1<\theta<\min(p-1,2)$, $k=1,2$,
\begin{equation}\label{eq:lbis}
\Big| \dot{{\ell}}_{k}+2\alpha{\ell}_{k}- (-1)^{k} \sigma \frac{z}{|z|} g(|z|)\Big|
\lesssim \|\vve\|_\ENE^{2}+\sum_{k=1,2}|{\ell}_{k}|^{2}+e^{-\theta |z|}.
\end{equation}

\item\emph{Equations of the exponential directions.} 
Let 
\begin{equation}\label{def:a}
a_k^{\pm} = \langle \vve,\vec Z_{k}^{\pm}\rangle.
\end{equation}
Then,
\begin{equation}\label{eq:a}
\left| \frac \ud{\ud t} a_k^{\pm}- \nu^\pm a_{k}^{\pm}\right|\lesssim \|\vve \|_\ENE^{2}+\sum_{k=1,2}|{\ell}_{k}|^{2}+ q(|z|).
\end{equation}
\end{enumerate}
\end{lemma}

\begin{remark}
We see on estimate~\eqref{eq:l} the damping of the Lorentz parameters $\ell_k$.
The more precise estimate~\eqref{eq:lbis} involves the nonlinear interactions which becomes preponderant for large time.
\end{remark}

\begin{proof}
Proof of (i).
The existence and uniqueness of the geometric parameters is proved for fixed time.
Let $0<\gamma\ll 1$.
First, for any $u\in H^1$ such that
\begin{equation}\label{on:u}
\inf_{|\xi_1 -\xi_2| > |\log\gamma|} \Big\|u-\sum_{k=1,2} \sigma_k Q(\cdot-\xi_k)\Big\|_{H^{1}}\leq \gamma,
\end{equation}
we consider $z_1(u)$ and $z_2(u)$ achieving the infimum
\begin{equation*}
\Big\|u-\sum_{k=1,2} \sigma_k Q(\cdot-z_k(u))\Big\|_{L^2}
=\inf_{|\xi_1 -\xi_2| > \frac{3}{4}|\log\gamma|} \Big\|u-\sum_{k=1,2} \sigma_k Q(\cdot-\xi_k)\Big\|_{L^2}.
\end{equation*}
Let
\begin{equation*}
\e(x) = u(x)- \sigma_1 Q(x-z_1(u))-\sigma_2 Q(x-z_2(u)),\quad
\|\e\|_{L^2}\leq \gamma.
\end{equation*}
By standard arguments since $|z_1(u) -z_2(u)| >\frac 34|\log\gamma|$, it holds
\begin{equation}\label{key4}
\langle \e, \partial_{x_j} Q(\cdot-z_1(u))\rangle= \langle \e, \partial_{x_j} Q(\cdot-z_2(u))\rangle=0.
\end{equation}
For $u$ and $\tilde u$ as in \eqref{on:u}, we compare the corresponding $z_k$, $\tilde z_k$ and
$\e$, $\tilde \e$.
First, for $\zeta$, $\tilde \zeta\in \RR^N$, setting $\check\zeta=\zeta-\tilde\zeta$, we observe the following estimates
\begin{align}
Q(\cdot-\zeta)-Q(\cdot-\tilde\zeta)
&=-(\check\zeta\cdot\nabla) Q(\cdot-\zeta)+O_{H^1}(|\check\zeta|^2), \label{eq:oo1}\\
\nabla Q(\cdot-\zeta)-\nabla Q(\cdot-\tilde\zeta)
&=-(\check\zeta\cdot\nabla^2) Q(\cdot-\zeta)+O_{H^1}(|\check\zeta|^2). \label{eq:oo2}
\end{align}
Thus, denoting $\check u=u-\tilde u$, $\check z_k=z_k-\tilde z_k$, $\check \e=\e-\tilde \e$,
we obtain
\begin{equation*}
\check u = \sum_{k=1,2} \sigma_k (\check z_k\cdot \nabla) Q(\cdot-z_k)
+\check \e + O_{H^1}(|\check z_1|^2+|\check z_2|^2).
\end{equation*}
(In the $O_{H^1}$, there is no dependence in $\check u$ or $\check \e$).
Projecting on each $\nabla Q(\cdot -z_k)$, using \eqref{key4} and the above estimates, we obtain
\begin{equation*}
|\check z_1|+|\check z_2|\lesssim 
\|\check u\|_{L^2} + (|\check z_1|+|\check z_2|)(e^{-\frac 12 |z|}+\|\tilde \e\|_{L^2} + |\check z_1|+|\check z_2|)
\end{equation*}
and thus, for $\gamma$, $\check z_1$ and $\check z_2$ small,
\begin{equation}\label{key1}
\|\check \e\|_{H^1} \lesssim \|\check u\|_{H^1}, \quad |\check z_1|+|\check z_2|\lesssim \|\check u\|_{L^2}.
\end{equation}
Therefore, for $\gamma$ small enough, this proves uniqueness and Lipschitz continuity of $z_1$ and $z_2$ with respect to $u$ in $L^2$.

Now, let $v\in L^2$ and $z_1$, $z_2$ be such that $\|v\|_{L^2}<\gamma$ and $| z_1 - z_2| >\frac 34 |\log \gamma|$. Set
\begin{equation*}
\eta(x) = v(x) + \sigma_1 (\ell_1\cdot \nabla) Q(x-z_1)+\sigma_2 (\ell_2 \cdot \nabla) Q(x-z_2).
\end{equation*}
Then, it is easy to check that the $2N$ conditions
\begin{equation}\label{ortho:v}
\langle \eta, \partial_{x_j} Q(\cdot-z_k)\rangle=0
\end{equation}
for $j=1,\ldots,N$, $k=1,2$,
are equivalent to a linear system in the components of $\ell_1$ and $\ell_2$ whose matrix a perturbation of the identity (for $\gamma$ small)  up to a multiplicative constant.
In particular, it is invertible and the existence and uniqueness of parameters $\ell_1(v,z_1,z_2),\ell_2(v,z_1,z_2)\in \RR^N$ satisfying
\eqref{ortho:v} and $|\ell_1|+|\ell_2|\lesssim \|v\|_{L^2}$ is  clear. Moroever, with similar notation as before, it holds
\begin{equation}\label{key2}
\|\check \eta\|_{L^2}+|\check \ell_1|+|\check \ell_2|\lesssim 
\|\check v\|_{L^2}+ |\check z_1|+|\check z_2|.
\end{equation}
Estimate~\eqref{eq:bound} is now proved.
In the rest of this proof, we formally derive the equations of $\vve$ and the geometric parameters from the equation of $u$. This derivation can be justified rigorously and used to prove by the Cauchy-Lipschitz theorem that the parameters are $\mathcal C^1$ functions of time
(see for instance \cite[Proof of~Lemma~2.7]{CM}).

Proof of (ii). First, by the definition of $\e$ and $\eta$,
\begin{equation*}
\partial_t\e =\partial_{t}u-\sum_{k=1,2}\partial_{t}Q_{k}
=\eta+\sum_{k=1,2}\left(\dot{z}_{k}-\ell_{k}\right)\cdot\nabla Q_{k}.
\end{equation*}
Second,
\begin{align*}
\partial_{t}\eta&=\partial_{tt}u+\sum_{k=1,2}\partial_{t}\left(\ell_{k}\cdot\nabla Q_{k}\right)\\
&=\Delta u-u+f(u)-2\alpha\partial_{t}u+\sum_{k=1,2}\dot{\ell}_{k}\cdot\nabla Q_{k}-\sum_{k=1,2}\left(\ell_{k}\cdot\nabla \right)\left(\dot{z}_{k}\cdot\nabla \right)Q_{k}.
\end{align*}
By~\eqref{def:ee}, $\Delta Q_{k}-Q_{k}+f(Q_{k})=0$ (from \eqref{eq:elliptic}) and the definition of $G$,
\begin{align*}
\Delta u-u+f(u)-2\alpha\partial_{t}u
&=\Delta\e-\e+f (R+\e )-f (R ) -2\alpha\eta\\
&\quad +2\alpha\sum_{k=1,2}\ell_{k}\cdot\nabla Q_{k}+G.
\end{align*}
Therefore,
\begin{align*}
\partial_t \eta & =
\Delta \e-\e+ f(R+\e)-f(R)-2\alpha\eta \\
&\quad +\sum_{k=1,2}\big(\dot{\ell}_{k}+2\alpha{\ell}_{k}\big)\cdot\nabla Q_{k}
-\sum_{k=1,2}({\ell}_{k}\cdot\nabla)(\dot{z}_{k}\cdot\nabla)Q_{k} + G.
\end{align*}

Proof of (iii)-(iv). We derive~\eqref{eq:z} from~\eqref{ortho}. For any $j=1,\ldots,N$, we have
\begin{equation*}
0=\frac \ud{\ud t}\langle\e,\partial_{x_j} Q_1 \rangle=\langle\partial_t \e,\partial_{x_j} Q_1 \rangle+\langle\e,\partial_{t} (\partial_{x_j} Q_1 )\rangle .
\end{equation*}
Thus~\eqref{syst_e} gives
\begin{align*}
\langle\eta, \partial_{x_j}Q_1\rangle+\langle \md_{\e}, \partial_{x_j}Q_1\rangle
-\langle\e, \dot{z_1}\cdot\nabla\partial_{x_j} Q_1\rangle=0.
\end{align*}
The first term is zero due to the orthogonality~\eqref{ortho}. Hence,
\begin{align*}
|\dot{z}_{1,j}-\ell_{1,j}|\|\partial_{x_j}Q\|_{L^2}^2\lesssim |\dot{z}_2-\ell_2|\int|\nabla Q_2(x)||\nabla Q_1(x)|\ud x
+|\dot{z}_1|\,\|\e\|_{L^2}.
\end{align*}
Thus, also using~\eqref{tech1} with $m=1$ and $m'=\frac 12$, we obtain
\begin{equation*}
|\dot{z}_1-\ell_1|\lesssim |\dot{z}_2-\ell_2|e^{-\frac 12 |z|}+|\dot{z}_1-\ell_1|\,\|\vve \|_\ENE+|\ell_1|\,\|\vve \|_\ENE.
\end{equation*}
Since $\|\vve \|_\ENE\lesssim \gamma$, this yields
\begin{equation*}
|\dot{z}_1-\ell_1|\lesssim |\dot{z}_2-\ell_2| e^{-\frac 12 |z|} +|\ell_1|\,\|\vve \|_\ENE.
\end{equation*}
Similarly, it holds
\begin{equation*}
|\dot{z}_2-\ell_2|\lesssim |\dot{z}_1-\ell_1|e^{-\frac 12 |z|} +|\ell_2|\,\|\vve \|_\ENE,
\end{equation*}
and thus, for large $|z|$,
\begin{equation*}
\sum_{k=1,2}|\dot{z}_{k}-{\ell}_{k}|\lesssim \left(|{\ell}_1 |+|{\ell}_2 |\right)\|\vve \|_\ENE,
\end{equation*}
which implies~\eqref{eq:z}.

Next, we derive~\eqref{eq:l}-\eqref{eq:lbis}. From~\eqref{ortho}, it holds
\begin{equation*}
0=\frac \ud{\ud t}\langle{\eta},\partial_{x_j} Q_1 \rangle
=\langle\partial_t \eta,\partial_{x_j} Q_1 \rangle
+\langle{\eta},\partial_{t} (\partial_{x_j} Q_1 )\rangle.
\end{equation*}
Thus, by~\eqref{ortho} and~\eqref{syst_e}, we have
\begin{align*}
0&=\langle\Delta\e-\e+f'(Q_1 )\e,\partial_{x_j} Q_1 \rangle
+\langle f(R+\e)-f(R)-f'(R)\e,\partial_{x_j} Q_1 \rangle\\
&\quad +\langle (f'(R)-f'(Q_1 ))\e, \partial_{x_j} Q_1 \rangle
\\
&\quad +\langle\md_{\eta}, \partial_{x_j} Q_1 \rangle+\langle G,\partial_{x_j} Q_1 \rangle
-\langle\eta, (\dot{z}_1 \cdot\nabla )\partial_{x_j} Q_1 \rangle.
\end{align*}
Since $\partial_{x_j} Q_1$ satisfies $\Delta\partial_{x_j} Q_1-\partial_{x_j} Q_1+f'(Q_1 )\partial_{x_j} Q_1=0$, the first term is zero. Next, by Taylor expansion (as $f$ is $\mathcal C^2$), we have
\[  f(R+\e)-f(R)-f'(R)\e = \e^2 \int_0^1 (1-\theta) f''(R+\theta \e) \ud \theta, \]
and by the $H^1$ sub-criticality of the exponent $p>2$, we infer
\begin{equation}\label{untard}
\left|\langle f(R+\e)-f(R)-f'(R)\e,\partial_{x_j} Q_1 \rangle\right|\lesssim \|\e\|_{H^1}^{2}.
\end{equation}
Then, again by Taylor expansion and $p> 2$,
\begin{equation*}
|f'(R)-f'(Q_1 )| |\partial_{x_j} Q_1 |
\lesssim  |Q_2| |Q_1|^{p-1} + |Q_1||Q_2|^{p-1} .
\end{equation*}
Thus, using also the Cauchy-Schwarz inequality and~\eqref{tech2},
\begin{equation}\label{deuxtard}
|\langle (f'(R)-f'(Q_1 ))\e, \partial_{x_j} Q_1 \rangle|
\lesssim \|\e\|_{L^2}^2+[q(|z|)]^2.
\end{equation}
Direct computations show that
\begin{align*}
\langle\md_{\eta}, \partial_{x_j} Q_1 \rangle&=
\big(\dot{\ell}_{1,j}+2\alpha\ell_{1,j}\big)\|\partial_{x_j}Q_1\|_{L^2}^2
+\langle(\dot{\ell}_2+2\alpha\ell_2)\cdot \nabla Q_2, \partial_{x_j}Q_1\rangle\\
&\quad-\sum_{k=1,2}\langle(\ell_k\cdot\nabla)(\dot{z}_k\cdot\nabla)Q_k, \partial_{x_j}Q_1\rangle.
\end{align*}
Thus, using Lemma~\ref{le:nonlin} and \eqref{eq:Q_sym}, for any $m\in (0,1)$,
\begin{align*}
\langle\md_{\eta}, \partial_{x_j} Q_1 \rangle
&= \big(\dot{\ell}_{1,j}+2\alpha\ell_{1,j}\big)\|\partial_{x_j}Q_1\|_{L^2}^2
\\&\quad
+O\left(\big|\dot{\ell}_2+2\alpha\ell_2\big|e^{-m |z|}\right)
+O\left(|\ell_2||\dot{z}_2| e^{-m |z|}\right).
\end{align*}
Using also~\eqref{eq:z}, it follows that
\begin{align*}
\langle\md_{\eta}, \partial_{x_j} Q_1 \rangle
&=\left(\dot{\ell}_{1,j}+2\alpha\ell_{1,j}\right)\|\partial_{x_j}Q_1\|_{L^2}^2
\\&\quad +O\left(\big|\dot{\ell}_2+2\alpha\ell_2\big| e^{-m |z|}\right)
+O\bigg(\sum_{k=1,2}|\ell_k|^2\bigg)+O\big(\|\vve \|_\ENE^2\big).
\end{align*}
By~\eqref{on:G} and the definition of $c_1$ in \eqref{on:g0}, we have, for any
$1<\theta<\min(p-1,2)$,
\begin{equation*}
 \left|\frac{\langle G,\partial_{x_j} Q_1 \rangle}{\|\partial_{x_1} Q\|_{L^2}^{2}} - \sigma\frac{z_j}{|z|}g(|z|)
 \right|\lesssim e^{-\theta |z|}.
\end{equation*}
Finally, by~\eqref{eq:z} again, we have
\begin{align*}
 |\langle\eta, (\dot{z}_1 \cdot\nabla )\partial_{x_j} Q_1 \rangle |
\lesssim (|\dot{z}_1-\ell_1|+|\ell_1|)\|\vve\|_\ENE
\lesssim \|\vve \|_\ENE^2+\sum_{k=1,2}|\ell_k|^2.
\end{align*}
Combining the above estimates, we have obtained
\begin{equation*}
\Big|\big(\dot\ell_1+2\alpha\ell_1\big)+\sigma\frac{z}{|z|}g(|z|)\Big|
\lesssim
\big|\dot{\ell}_2+2\alpha\ell_2\big|e^{-m|z|} +\sum_{k=1,2}|\ell_k|^2+\|\vve\|_\ENE^2+ e^{-\theta |z|} .
\end{equation*}
Similarly, from $(\eta, \partial_{x_j}Q_2)=0$, we check
\begin{equation*}
\Big|\big(\dot\ell_2+2\alpha\ell_2\big)- \sigma\frac{z}{|z|}g(|z|)\Big|
 \lesssim
\big|\dot{\ell}_1+2\alpha\ell_1\big|e^{-m |z|}  +\sum_{k=1,2}|\ell_k|^2+\|\vve\|_\ENE^2+ e^{-\theta |z|}.
\end{equation*}
These estimates imply~\eqref{eq:lbis};~\eqref{eq:l} follows readily using~\eqref{on:gg}.

Proof of (v). By~\eqref{syst_e}, we have
\begin{align*}
\frac \ud{\ud t}a_1 ^{\pm} & = \langle\partial_{t}\vve,\vec{Z}_1 ^{\pm}\rangle
+\langle\vve,\partial_{t}\vec{Z}_1 ^{\pm}\rangle\\
&=(\zeta^\pm-2\alpha) \langle \eta , Y_1\rangle+\langle \Delta \e-\e+f'(Q_1) \e,Y_1\rangle
 \\
&\quad + \langle f(R+\e)-f(R)-f'(R)\e, Y_1\rangle
+ \langle (f'(R)-f'(Q_1))\e, Y_1\rangle\\
&\quad +\langle G,Y_1\rangle+\zeta^\pm \langle \md_\e,Y_1\rangle
+\langle \md_\eta,Y_1\rangle-\langle\vve, \dot{z}_1 \cdot \nabla\vec{Z}_1 ^{\pm}\rangle.
\end{align*}
Using $\zeta^\pm-2\alpha=\nu^\pm$
and~\eqref{def:a}, $\LL Y=-\nu_0^2 Y$ and $\nu_0^2 = \nu^\pm \zeta^\pm$,
we observe that
\begin{equation*}
(\zeta^\pm-2\alpha) \langle \eta , Y_1\rangle+\langle \Delta \e-\e+f'(Q_1) \e,Y_1\rangle
= \nu^\pm a_1^\pm.
\end{equation*}
Using the decay properties of $Y$ in Lemma~\ref{le:L} and proceeding as before for~\eqref{untard} and~\eqref{deuxtard},
\begin{equation*}
|\langle f(R+\e)-f(R)-f'(R)\e, Y_1\rangle|+|\langle (f'(R)-f'(Q_1))\e, Y_1\rangle|
\lesssim \| \e\|_{L^2}^2 + e^{-\frac 32 |z|}.
\end{equation*}
Next, by \eqref{tech2}, $|\langle G,Y_1\rangle|\lesssim q(|z|)$.
Last, by~\eqref{eq:z} and~\eqref{eq:l},
\begin{equation*}
|\langle \md_\e,Y_1\rangle|+|\langle \md_\eta,Y_1\rangle|
+|\langle\vve, \dot{z}_1 \cdot \nabla\vec{Z}_1 ^{\pm}\rangle|
\lesssim \|\vve\|_\ENE^2 + \sum_{k=1,2} |\ell_k|^2 + q(|z|).
\end{equation*}
Gathering these estimates, and proceeding similarly for $a_2^\pm$,~\eqref{eq:a} is proved.
\end{proof}

\subsection{Energy estimates}
For $\mu>0$ small to be chosen, we denote $\rho=2\alpha-\mu$. Consider the nonlinear energy functional
\begin{equation} \label{def:E}
\mathcal{E} =
\int \big\{ |\nabla\e|^2+ (1-\rho\mu )\e^{2}
+(\eta+\mu\e)^2- 2 [F (R +\e )-F (R )-f (R )\e]\big\}.
\end{equation}

\begin{lemma}\label{le:ener} There exists $\mu>0$ such that
in the context of Lemma~\ref{le:dec},
 the following hold.
\begin{enumerate}
\item\emph{Coercivity and bound.}
\begin{equation}\label{eq:coer}
\mu \|\vve \|_\ENE^{2}-\frac{1}{2\mu}\sum_{k=1,2}\left( (a_{k}^{+})^{2}+(a_{k}^{-})^{2}\right)
\leq \mathcal{E}\leq \frac 1\mu \|\vve \|_\ENE^{2}.
\end{equation}
\item\emph{Time variation.}
\begin{equation}\label{eq:E}
\frac \ud{\ud t}\mathcal{E}\le -2\mu\mathcal{E}
+\frac{1}{\mu}\|\vve \|_\ENE \bigg[\|\vve \|_\ENE^{2}+\sum_{k=1,2}|\ell_{k}|^{2}+q(|z|)\bigg].
\end{equation}
\end{enumerate}
\end{lemma}
\begin{remark}
The above lemma is valid for any small enough $\mu>0$. For future needs, we assume further
\begin{equation}\label{on:mu}
\mu \leq \min\left(1,\alpha,  |\nu_-|\right).
\end{equation}
One checks that a usual linearized energy, corresponding to $\mu=0$ in the definition of $\mathcal{E}$, would only gives damping for the component $\eta$. This is the reason why we introduce the modified energy $\mathcal{E}$.
For the simplicity of notation, the same small constant $\mu>0$ is used in~\eqref{eq:coer} and~\eqref{eq:E} though in the former estimate the small constant is related to the coercivity constant $c$ of Lemma~\ref{le:L}, while in the latter it is related to the damping $\alpha$.
\end{remark}

\begin{proof} Proof of (i).
The upper bound on $\mathcal{E}$ in \eqref{eq:coer} easily follows from the energy subcriticality of $p$.
The coercivity is proved for fixed time and so we omit the time dependency. By translation invariance, we
 assume without loss of generality that $z_1=-z_2=\frac z2$.
Let $\chi:\RR\to\RR$ be a smooth function satisfying the following properties
\begin{equation*}
\chi=1 \text{ on } \left[0,1\right], \quad \chi=0 \text{ on } [2,+\infty), \quad \chi'\leq 0 \text{ on } \RR.
\end{equation*}
For $\lambda=\frac {|z|}4\gg 1$, let
\begin{equation*}
\chi_1(x)=\chi\left(\frac {\left|x-z_{1}\right|}{\lambda}\right),\quad
\chi_2(x)=(1-\chi_1^2(x))^\frac 12.
\end{equation*}
We define $\e_k=\e(\cdot+z_k)\chi_k(\cdot+z_k)$ for $k=1,2$, so that
\begin{equation*}
\int |\nabla \e_k|^2=\int |\nabla \e|^2\chi_k^2-\int \e^2 (\Delta \chi_k)\chi_k,
\end{equation*}
and thus
\begin{equation}\label{eq:2s}
\int |\nabla \e|^2 = \int |\nabla \e_1|^2 +|\nabla \e_2|^2 + O(\lambda^{-2} \|\e\|_{L^2}^2).
\end{equation}

Next, using~\eqref{ortho},
\begin{equation*}
\langle \e_k,\partial_{x_j} Q\rangle
=\sigma_k \langle \e \chi_k,\partial_{x_j} Q_k\rangle
=\sigma_k \langle \e , (\chi_k-1)\partial_{x_j} Q_k\rangle
=O(e^{-\frac {|z|} 4} \|\e\|_{L^2} ),
\end{equation*}
and
\begin{equation*}
\langle \e_k , Y\rangle=\sigma_k \langle \e \chi_k, Y_k\rangle
=\sigma_k \langle \e , Y_k\rangle + O(e^{-\frac {|z|} 4} \|\e\|_{L^2} ).
\end{equation*}
Thus, applying (ii) of Lemma~\ref{le:L} to $\e_k$, one obtains
\begin{equation*}
\langle \LL \e_k, \e_k\rangle
\geq c \|\e_k\|_{H^1}^2 - C \langle \e , Y_k\rangle^2 - C (e^{-\frac {|z|} 4}+\lambda^{-2}) \|\e\|_{L^2}^2.
\end{equation*}
Using~\eqref{eq:2s}, we obtain for $z$ large
\begin{equation*}
\langle \LL \e, \e\rangle
\geq \frac c 2 \|\e\|_{H^1}^2 - C \sum_{k=1,2}\langle \e , Y_k\rangle^2 .
\end{equation*}
The estimate
\begin{equation*}
\int \big\{ |\nabla\e|^2+ \e^{2}
+\eta^2- f' (R) \e^2 \big\}\geq \mu \|\vve\|_{H^1\times L^2}^2
- C (a_k^+)^2 - C (a_k^-)^2
\end{equation*}
then follows from
\begin{equation*}
\langle \e,Y_k\rangle = \frac {\langle\vve,\vec{Z}_k^+ - \vec{Z}_k^-\rangle }{\zeta^+-\zeta^-}
 = \frac {a_k^+ - a_k^-}{\zeta^+-\zeta^-}.
\end{equation*}
Using
\begin{equation*}
|F (R +\e )-F (R )-f (R )\e - f'(R) \e|\lesssim |\e|^3+|\e|^{p+1},
\end{equation*}
and $\int |\e|^{3}+ |\e|^{p+1}\lesssim\|\e\|_{H^{1}}^{3}+ \|\e\|_{H^1}^{p+1}$ by energy subcriticality of $p$,
for $\gamma$ and $\mu$ small, we have proved \eqref{eq:coer}.

Proof of (ii).
From direct computations and integration by parts, we have
\begin{align*}
\frac 12 \frac \ud{\ud t} \mathcal{E}
& =\int \partial_t \e \left[- \Delta \e
+ (1-\rho\mu ) \e -f(R+\e)+f(R)\right]
+ (\partial_t \eta+\mu\partial_t \e) (\eta+\mu\e )\\
&\quad +\int \sum_{k=1,2} (\dot{z}_{k}\cdot\nabla Q_{k})
\left[f(R+\e)-f(R)-f'(R)\e\right]=\mathbf{g_1 }+\mathbf{g_2 }.
\end{align*}
Using~\eqref{syst_e}, integration by parts and $2\alpha=\rho+\mu$, we compute
\begin{align*}
\mathbf{g_1 }&=
-\mu\int\big\{|\nabla \e|^2+(1-\rho\mu)\e^2 -\e [f(R+\e)-f(R)]\big\}
-\rho \int (\eta+\mu\e )^2\\
&\quad+\int {\rm{Mod}_{\e}} \big\{-\Delta{\e}+ (1-\rho\mu ){\e}- \big[f(R+\e)-f(R)\big]\big\}\\
&\quad+\int (\eta+\mu\e ) [\md_{\eta}+\mu\md_{\e}] +\int (\eta+\mu\e ) G
\\&=\mathbf{g_{1,1}}+\mathbf{g_{1,2}}+\mathbf{g_{1,3}}+\mathbf{g_{1,4}}.
\end{align*}
Note that by $0<\mu<\alpha$, one has $\rho-\mu = 2(\alpha-\mu)>0$ and so
\begin{align*}
\mathbf{g_{1,1}}&= - \mu \mathcal{E}- 2\mu\int \big[F(R+\e)-F(R)-f(R)\e-\tfrac 12f'(R)\e^{2}\big]\\
&\quad +\mu\int \e [f(R+\e)-f(R)-f'(R)\e ]-(\rho-\mu)\int (\eta+\mu\e)^{2}
\le -\mu \mathcal{E}+C \|\e\|_{H^1}^{3} ,
\end{align*}
where we have estimated, using $p>2$, H\"older inequality, the sub-criticality of $p$ and Sobolev embedding,
\begin{align*}
&\int \left|F(R+\e)-F(R)-f(R)\e- \tfrac 12 f'(R)\e^{2}\right|\\
&\quad +\int |\e [f(R+\e)-f(R)-f'(R)\e]|
\lesssim \int |\e|^3 |R|^{p-2} + |\e|^{p+1}\lesssim \|\e\|_{H^1}^{3}.
\end{align*}
Using the Cauchy-Schwarz inequality and~\eqref{eq:z}-\eqref{eq:l}, we also derive the following estimates
\begin{equation*}
|\mathbf{g_{1,2}} | \lesssim
\left( |\dot z_1-\ell_1|+|\dot z_2-\ell_2|\right) \|\e \|_{H^1}
\lesssim  \|\vve \|_\ENE \bigg[\|\vve \|_\ENE^{2}+\sum_{k=1,2}|\ell_{k}|^{2}\bigg],
\end{equation*}
and (also using the orthogonality conditions~\eqref{ortho})
\begin{align*}
|\mathbf{g_{1,3}} | &
=\bigg|\int (\eta+\mu \e) \big(\sum_{k=1,2} (\ell_k\cdot\nabla)(\dot z_k\cdot\nabla)Q_k\big)\bigg|\\
& \lesssim (|\ell_1|+|\ell_2|)(|\dot z_1-\ell_1|+|\dot z_2-\ell_2|+|\ell_1|+|\ell_2|)(\|\e \|_{L^2}+\|\eta\|_{L^2})\\
&\lesssim   \|\vve \|_\ENE \bigg[\|\vve \|_\ENE^{2}+\sum_{k=1,2}|\ell_{k}|^{2}\bigg] .
\end{align*}
Next, from~\eqref{tech2} and the Cauchy-Schwarz inequality,
\begin{equation*}
|\mathbf{g_{1,4}}|\lesssim \|G\|_{L^2} (\|\e \|_{L^2}+\|\eta\|_{L^2})
\lesssim q(|z|) \|\vve \|_\ENE.
\end{equation*}
Last, by~\eqref{eq:z}, proceeding as before, we see that
\begin{align*}
 |\mathbf{g_2 } |
 &\lesssim (|\dot z_1|+|\dot z_2|) \| \e \|_{H^1}^2
 \lesssim (|\dot z_1-\ell_1|+|\dot z_2-\ell_2|+|\ell_1|+|\ell_2|) \|\vve \|_\ENE^2\\
 &\lesssim \|\vve \|_\ENE^2 \bigg[\|\vve \|_\ENE^{2}+\sum_{k=1,2}|\ell_{k}|\bigg].
\end{align*}
Gathering the above estimates, \eqref{eq:E} is proved, taking $\mu$ small enough.
\end{proof}

\subsection{Trichotomy of evolution}\label{S:2.4}
Estimates~\eqref{eq:z},~\eqref{eq:l},~\eqref{eq:lbis},~\eqref{eq:a} and~\eqref{eq:E} give basic information on the evolution of the various components of a solution in the framework of the decomposition introduced in Lemma~\ref{le:dec}.

We introduce notation related to modified parameters that allow us to justify the following trichotomy in the evolution of the solution.
\begin{description}
\item[ODE behavior for $|z|$] The distance $z=z_1-z_2$ formally satisfies
\begin{equation*}
\frac \ud{\ud t}\bigg[\frac 1{q(|z|)}\bigg]= -\frac{\sigma g_0}{\alpha}.
\end{equation*}
Note that  for $\sigma=-1$, $\log t$ is an approximate solution of this ODE.
This justifies the rigidity results in Theorem~\ref{th:2}.
In contrast, when $\sigma=1$, there are no solution such that $|z|\to \infty$, which explains the non-existence result in Theorem~\ref{th:1}.
\item[Exponential growth] The parameters $a^+_k$ are related to the forward exponential instability of the solitary wave. They will require a specific approach, involving backward in time arguments. The existence of exactly one direction of instability for each solitary wave justifies Theorem~\ref{th:3}.
\item[Damped evolution] The parameters $\ell_k$, $a^-_k$ and the remainder $\vve$ (without its unstable components $a_k^+$) enjoy exponential
damping;  they will be easily estimated provided that the other parameters are locked, see Proposition~\ref{le:BS}.
\end{description}

First, we set
\begin{equation}
y = z + \frac{\ell}{2\alpha},\quad r = |y|.
\end{equation}
Second, we define
\begin{equation}
b=\sum_{k=1,2} (a_k^+)^2.
\end{equation}
Third, we introduce notation for the damped components:
\begin{equation}
\mathcal F = \mathcal E + \BB,\quad \BB= \sum_{k=1,2} |\ell_k|^2 + \frac 1{2 \mu} \sum_{k=1,2} (a_k^-)^2 ,
\end{equation}
and for all the components of the solution
\begin{equation}
\mathcal N = \bigg[ \|\vve\|_\ENE^2 + \sum_{k=1,2} |\ell_k|^2 \bigg]^{\frac 12}.
\end{equation}
Last, we define
\begin{equation}
\mathcal M= \frac 1{\mu^2}\left(\mathcal F - \frac{b}{2\nu^+}\right).
\end{equation}

In the following lemma, we rewrite the estimates of Lemmas~\ref{le:dec} and~\ref{le:ener}  in terms of these new parameters.

\begin{lemma}\label{le:new}
In the context of Lemma~\ref{le:dec}, the following hold.
\begin{enumerate}
\item \emph{Comparison with original variables.}
\begin{gather}
 \big| r-|z| \big|\leq |y-z| \lesssim \mathcal N,\quad
 |g(|z|)-g_0 q(r)|\lesssim q(r) (\mathcal N+r^{-1}), \label{eq:new1}\\
 \mu \mathcal N^2\leq 
 \mu \|\vve\|_\ENE^2 + \sum_{k=1,2} |\ell_k|^2 \leq \mathcal F + \frac b{2\mu}
\lesssim\mathcal N^2.\label{eq:new2}
\end{gather}
\item\emph{ODE behavior for the distance.} For some $K>0$,
\begin{equation}\label{eq:dist}
\bigg|\frac \ud{\ud t} \bigg[ \frac 1{q(r)}\bigg] + \frac{\sigma g_0}{\alpha} \bigg|
\leq \frac{K}{q(r)}\left( \mathcal N^2+r^{-1}q(r)\right).
\end{equation}
\item\emph{Exponential instability.}
\begin{equation}\label{eq:b}
 | \dot b - 2 \nu^+ b |\lesssim \mathcal N^3+q(r) \mathcal N .
\end{equation}
\item\emph{Damped components.}
\begin{equation}\label{eq:damped}
\frac \ud{\ud t}\mathcal F + 2 \mu \mathcal F \lesssim \mathcal N^3+q(r) \mathcal N ,\quad
\frac \ud{\ud t}\BB + 2 \mu \BB \lesssim \mathcal N^3+q(r) \mathcal N .
\end{equation}
\item\emph{Liapunov type functional.} 
\begin{equation} \label{eq:Mbis}
\frac \ud{\ud t} \mathcal M \leq -\mathcal N^2 + C [q(r)]^2.
\end{equation}
\item\emph{Refined estimates for the distance.}
Setting
\begin{equation*}
R^+=\frac1{q(r)} \exp\left( K\mathcal M \right) \quad \text{and}\quad
R^-=\frac1{q(r)} \exp\left(- K\mathcal M \right),
\end{equation*}
($K$ is given in \eqref{eq:dist}), it holds
\begin{align}
\frac{\ud}{\ud t} R^+ &\leq \left(-\frac{\sigma g_0}{\alpha}+2Kr^{-1}\right)\exp\left( K\mathcal M \right),\label{eq:Rp}\\
\frac{\ud}{\ud t} R^- &\geq \left(-\frac{\sigma g_0}{\alpha}-2Kr^{-1}\right)\exp\left( -K\mathcal M \right).\label{eq:Rm}
\end{align}
\end{enumerate}
\end{lemma}
\begin{proof}
Proof of~\eqref{eq:new1}. It follows
from the triangle inequality that
\begin{equation*}
\big|r-|z| \big| \leq |y-z|\leq \frac{|\ell|}{2\alpha }\lesssim \mathcal N.
\end{equation*}
The second part of~\eqref{eq:new1} then follows from~\eqref{on:gg}.

Proof of~\eqref{eq:new2}. It follows readily from~\eqref{eq:coer}.

Proof of~\eqref{eq:dist}. First, from~\eqref{eq:z},~\eqref{eq:lbis} and~\eqref{eq:new1}, we note ($\theta$ is as in \eqref{eq:lbis})
\begin{equation}\label{for:y}\begin{aligned}
\dot y = \dot z + \frac{\dot \ell}{2\alpha}
&=-\frac\sigma{\alpha} \frac{z}{|z|} g(|z|)+O(\mathcal N^2+e^{-\theta |z|})\\
&=-\frac{\sigma g_0}{\alpha} \frac{y}{r} q(r)+O(\mathcal N^2+r^{-1}q(r)).
\end{aligned}\end{equation}
Hence,
\begin{equation*}
\dot r = \frac{\dot y\cdot y}{r}
=-\frac{\sigma g_0}{\alpha} q(r)+O(\mathcal N^2+r^{-1}q(r)).
\end{equation*}
Using also $|q'(r)+q(r)|\lesssim r^{-1} q(r)$ (from~\eqref{Qdec}), we find
\begin{equation*}
\frac \ud{\ud t} \bigg[ \frac 1{q(r)}\bigg]
=-\frac{\dot r q'(r)} {[q(r)]^2} = - \frac{\sigma g_0}{\alpha} + \frac 1{q(r)} O ( \mathcal N^2 +r^{-1}q(r) ).
\end{equation*}

Proof of~\eqref{eq:b}. It follows from~\eqref{eq:a} and $|a^+_k|\lesssim \|\vve\|_\ENE
\leq \mathcal N$.

Proof of~\eqref{eq:damped}. 
From the expression of $\mathcal F$ and then~\eqref{eq:E},~\eqref{eq:l} and~\eqref{eq:a}\begin{align*}
\frac \ud{\ud t} \mathcal F
&=\frac \ud{\ud t} \mathcal E + 2 \sum_{k=1,2} \dot \ell_k \cdot\ell_k +\frac 1\mu \sum_{k=1,2} \dot a_k^- a_k^- \\
& \leq -2\mu \mathcal E -4\alpha \sum_{k=1,2} |\ell_k|^2 + \frac{\nu^-}\mu \sum_{k=1,2} (a_k^-)^2
+ O(\mathcal N^3+q(r) \mathcal N).
\end{align*}
Since $0<\mu<\alpha$ and $0<\mu<|\nu^-|$ (see~\eqref{on:mu}) we obtain~\eqref{eq:damped} for $\mathcal F$.
The proof for $\BB$ is the same.

Proof of \eqref{eq:Mbis}.
First, it follows from combining~\eqref{eq:b} and~\eqref{eq:damped} that
\begin{equation*}
\mu^2 \frac \ud{\ud t} \mathcal M \leq -2\mu \left(\mathcal F + \frac b{2\mu}\right) + O(\mathcal N^3+q(r)\mathcal N),
\end{equation*}
Thus, from~\eqref{eq:new2},
we observe that
\begin{equation*}
\mathcal N^2
\leq \frac 1{\mu} \left(\mathcal F+\frac{b}{2\mu}\right)
\leq -\frac 1{2} \frac \ud{\ud t} \mathcal M + O(\mathcal N^{3}+q(r)\mathcal N).
\end{equation*}
and \eqref{eq:Mbis} follows for $\mathcal N$ small enough.

Proof of \eqref{eq:Rp}-\eqref{eq:Rm}.
It follows from direct computation and \eqref{eq:dist} and \eqref{eq:Mbis} that
(for $r$ large enough)
\begin{equation*}
\exp\left(-K\mathcal M\right) \frac{\ud}{\ud t} R^+
=\frac{\ud}{\ud t} \left[\frac 1{q(r)}\right] + \frac K{q(r)} \frac{\ud}{\ud t} \mathcal M
\leq -\frac {\sigma g_0}{\alpha} + 2K r^{-1}.
\end{equation*}
The computation for $\frac{d}{dt}R^-$ is similar.
\end{proof}

\subsection{Energy of a $2$-solitary wave}
We observe from the definition and the energy property~\eqref{energy} that
a $2$-solitary wave $u$ of~\eqref{nlkg} satisfies
\begin{equation}\label{suite0}
\lim_{t\to \infty} E({\vec{u}}(t)) = \int |\nabla Q|^2+Q^2-2F(Q) = 2E(Q,0).
\end{equation}
More precisely, we expand the energy for a solution close to a $2$-solitary wave.
\begin{lemma}
In the context of Lemma~\ref{le:dec}, the following holds
\begin{equation}\label{eq:E:R}
E(\vec u)=2E(Q,0)-\sigma c_1 g_0 q(r) +O(r^{-1}q(r))+O(q(r)\mathcal{N})+O(\mathcal{N}^2).
\end{equation}
\end{lemma}
\begin{proof}
Expanding $E(u,\partial_t u)$ using the decomposition~\eqref{def:ee}, integration by parts,
the equation $-\Delta Q_k+Q_k-f(Q_k)=0$ and the definition of $G$ in~\eqref{def:G}, we find
\begin{align*}
2 E(u,\partial_{t}u)
&= \int \left|\partial_{t}u\right|^{2} + 2 E\left(R,0\right)
-2\int G\e \\
&\quad
+ \int \left(|\nabla\e|^2 + \e^2 -2 F(R+\e)+2 F(R)+2f(R)\e\right).
\end{align*}
Thus, using \eqref{tech2}, the subcriticality of $p$ and Sobolev embedding, there hold
\begin{equation*}
2 E(u,\partial_{t}u)
=\int \left|\partial_{t}u\right|^{2} + 2 E\left(R,0\right)
+O \big(q(r)\|\vve\|_\ENE+\|\vve\|_\ENE^2 \big).
\end{equation*}
Note that $\partial_t u = \eta-\sum_{k=1,2} (\ell_k\cdot \nabla) Q_k $ implies $\|\partial_t u\|_{L^2}\lesssim \mathcal N$.
Next, by direct computation, $-\Delta Q_1+Q_1-f(Q_1)=0$ and then \eqref{new3},~\eqref{on:gg} and \eqref{tech3}
\begin{align*}
E(Q_1 +Q_2 ,0)
&=2E(Q,0)+\int \left(\nabla Q_1 \cdot\nabla Q_2 +Q_1 Q_2 -f(Q_1 )Q_2 -f(Q_2 )Q_1 \right) \\
&\quad-\int \left(F(R)-F(Q_1 )-F(Q_2 )-f(Q_1 )Q_2 -f(Q_2 )Q_1 \right) \\
&=2E(Q,0)-\int f(Q_2) Q_1+ O (e^{-\frac 54|z|}) \\
&=2E(Q,0)- \sigma  c_1g(|z|) + O(|z|^{-1} g(|z|)).
\end{align*}
Using \eqref{eq:new1}, 
$g(|z|) =g_0 q(r) +   O(r^{-1} q(r)) +O(q(r)\mathcal N )$
and the proof is complete.
\end{proof}

As a consequence of~\eqref{energy}, \eqref{suite0} and \eqref{eq:E:R}, we obtain the following estimate.
\begin{corollary}\label{cor:ener}
Let $\vec u$ be a $2$-solitary wave solution of~\eqref{nlkg} satisfying the decomposition of Lemma~\ref{le:dec} on $[t,\infty)$, for some $t\in \RR$.
Then,
\begin{equation}\label{suite}
\int_{t}^{+\infty} \|\partial_{t} u(s)\|_{L^2}^2 \ud s \lesssim q(r(t)) + \mathcal N^2(t).
\end{equation}
\end{corollary}

\section{Proofs of Theorems~\ref{th:1} and \ref{th:2}}\label{S:3}

\subsection{General estimates}\label{S:2.5}

\begin{proposition}\label{le:BS}
There exists $C>0$ and $\delta_{1}>0$ such that the following hold.
For any $0<\delta<\delta_{1}$ and any $2$-solitary wave $\vec u$ of~\eqref{nlkg}, there exists  $\TD\in \RR$ such that
$\vec u(t)$ admits a decomposition as in Lemma~\ref{le:dec} in a neighborhood of $\TD$ with
\begin{equation}\label{at:T1}
\mathcal N(\TD)\leq \delta\quad \hbox{and}\quad q(r(\TD))\leq \delta^2.
\end{equation}
Moreover, for any such $T_{\delta}$, it holds:
\begin{description}
\item[Same sign case] if $\sigma=1$ then there exists $\TS>\TD$ such that $\vec u$ admits a decomposition as in Lemma~\ref{le:dec} on $[\TD,\TS]$ 
with
\begin{equation}\label{eq:pr:th:1}
\forall t \in [\TD,\TS], \quad\mathcal N(t)\leq C \delta \quad \text{and}\quad q(r(\TS))=\delta^\frac 32.
\end{equation}

\item[Opposite sign case] if $\sigma=-1$ then $\vec u$ admits a decomposition as in Lemma~\ref{le:dec} for all~$t\geq \TD$ and satisfies
\begin{equation}\label{eq:afaire2}
\forall t \in [\TD,\infty), \quad  \mathcal N(t) \leq C \delta
\quad \text{and} \quad
\bigg|\frac 1{q(r(t))}-\frac{g_0}{\alpha} t\bigg|\lesssim \frac t{|\log \delta|} + \tilde C
\end{equation}
for some $\tilde C>0$.
\end{description}
\end{proposition}
\begin{proof}
Let $\vec u$ be a $2$-solitary wave of \eqref{nlkg}.
For $\delta$ to be taken small enough, the existence of $\TD$ satisfying \eqref{at:T1} is a consequence of Definition~\ref{def:2sol}
and Lemma~\ref{le:dec}.

For a constant $C>1$ to be taken large enough, we introduce the following bootstrap estimate
\begin{equation}\label{BS:1}
\mathcal N\leq C \delta,\quad
q(r)\leq \delta^{\frac 32},\quad
b\leq C \delta^2.
\end{equation}
Set
\begin{equation*}
\TS=\sup\big\{t\in[\TD,\infty) \text{ such that \eqref{BS:1} holds on $[T_{\delta},t]$} \big\}>T_{\delta}.
\end{equation*}
In the remainder of the proof, the implied constants in $\lesssim$ or $O$ \emph{do not depend} on the constant $C$ appearing in the bootstrap assumption \eqref{BS:1}. We start by improving the bootstrap assumption on $\mathcal N$ and $b$.

\bigskip

\emph{Estimate on $\mathcal N$.}
We now improve the bootstrap assumption \eqref{BS:1} on $\mathcal N$.

From~\eqref{eq:damped} and \eqref{BS:1}, it holds on $[\TD,\TS)$,
\begin{equation*}
\frac \ud{\ud t} \left[ e^{2\mu t} \mathcal F\right] \lesssim C^3 \delta^3 e^{2\mu t} + C \delta^{\frac 5 2} e^{2\mu t} \lesssim
 \delta^2 e^{2\mu t},
 \end{equation*}
for $\delta>0$ small enough. From~\eqref{eq:new2} and \eqref{at:T1}, we have $\mathcal F(\TD)\lesssim \delta^2$. Thus, integrating the above estimate on $[\TD,t]$,
for any $t\in [\TD,\TS)$, it holds $\mathcal F\lesssim \delta^2$.
In particular, by \eqref{eq:new2}, we obtain
\begin{equation}\label{bnf2}
\|\vve\|_\ENE^2 \lesssim \mathcal F+b
\lesssim C \delta^2.
\end{equation}
Arguing similarly for the quantity $\BB$, we have
\begin{equation}\label{bnf3}
\sum_{k=1,2} |\ell_k|^2 + \sum_{k=1,2} (a_k^-)^2\lesssim \BB \lesssim \delta^{2}.
\end{equation}
Hence we obtain
\[ \forall t \in [\TD, \TS), \quad \mathcal N(t) \lesssim \sqrt{C} \delta. \]
For $C$ large enough, this strictly improves the estimate \eqref{BS:1} of~$\mathcal N$ on $[\TD,\TS)$.

\bigskip

\emph{Estimate on $b$.}
We now prove that for $C$ large enough, it holds 
\[ \forall t \in [\TD,\TS), \quad b(t) \le \frac C2 \delta^2. \]
From \eqref{suite} in Corollary~\ref{cor:ener} , we have
\begin{equation}\label{int1.3}
\int_{\TD}^{+\infty} \|\partial_{t} u(s)\|_{L^2}^2 \ud s \lesssim q(r(\TD)) + \mathcal N^2(\TD)\lesssim \delta^2.
\end{equation}
By \eqref{at:T1}, we have $b(\TD)\lesssim \delta^2$.

For the sake of contradiction, take $C$ large and assume that there exists $t_2\in [\TD,\TS)$ such that
\begin{equation*}
b(\Td)= \frac C2 \delta^2,\quad b(t)< \frac C2 \delta^2 \quad \text{on } [\TD,t_{2}).
\end{equation*}
On the one hand, by continuity of $b$, there exists $\Tu\in [\TD,t_{2})$ such that
\begin{equation*}
b(\Tu)=\frac C4  \delta^2\quad \text{and}\quad
b(t)> \frac C4 \delta^2 \quad \text{on } (\Tu,t_{2}].
\end{equation*}
Using~\eqref{eq:b} and~\eqref{BS:1}, we have
\begin{equation*}
\frac \ud{\ud t} b = 2\nu^+ b + O(C^3 \delta^3 + C \delta^{\frac 52})
\end{equation*}
which implies (for $\delta$ small enough with respect to $1/C$)
\begin{equation}\label{bnf5}
\Td-\Tu=\frac{\log 2}{2\nu^+}+O(\delta^{\frac{1}{2}}),
\end{equation}
and thus
\begin{equation}\label{bnf6}
\int_{\Tu}^{\Td} b \gtrsim C \delta^2.
\end{equation}

On the other hand, by \eqref{def:gs}, \eqref{def:ee},
$\partial_t u=\eta-\sum_{k=1,2} (\ell_k\cdot\nabla) Q_k$ and~\eqref{bnf3} we have
$\|\eta\|_{L^2}^2 \lesssim \|\partial_t u\|_{L^2}^2 + \delta^{2}$ on $[\TD,\TS]$.
Thus, using \eqref{int1.3} and \eqref{bnf5},
\begin{equation}\label{bnf7}
\int_{\Tu}^{\Td}  \|\eta(t)\|_{L^2}^2 \ud t \lesssim \delta^{2}.
\end{equation}
By the definition of $a_k^\pm$, one has
\begin{equation*}
a_k^+=\zeta^+\langle \e,Y_k\rangle+\langle \eta,Y_k\rangle, \quad
a_k^-=\zeta^-\langle \e,Y_k\rangle+\langle \eta,Y_k\rangle
\end{equation*}
and thus
\begin{equation*}
a_k^+ = \frac{\zeta^+}{\zeta^-} a_k^- + \frac{\zeta^- -\zeta^+}{\zeta^-} \langle \eta,Y_k\rangle.
\end{equation*}
Combining \eqref{bnf3}, \eqref{bnf5} and \eqref{bnf7}, we find
$\displaystyle \int_{\Tu}^{\Td} b \lesssim \delta^{2}$,
a contradiction with~\eqref{bnf6} for $C$ large enough.

\bigskip

We can now prove the two statements of Proposition \ref{le:BS} themselves. First observe that~\eqref{bnf2} and~\eqref{bnf3} gives on $[\TD,\TS)$
\[ \exp(\pm K \mathcal M) = 1 + O(C \delta^2). \]

\emph{Same sign case.}
We use the quantity $R^+$ defined in Lemma~\ref{le:new}.
From~\eqref{eq:Rp}, for $r$ large,
\begin{equation*}
\frac{\ud}{\ud t} R^+\leq \left(-\frac{g_0}{\alpha}+2Kr^{-1}\right)\exp\left( K\mathcal M \right)
\leq-\frac{g_0}{2\alpha}.
\end{equation*}
Therefore, for all $t\in [\TD,\TS]$, it holds
\begin{equation*}
R^+(t)\leq R^+(\TD) -\frac{g_0}{2\alpha} ( t-\TD).
\end{equation*}
Assuming $\TS=\infty$, $R^+(t)$ becomes negative for some time, which is absurd.
Since all the estimates in \eqref{BS:1} have been strictly improved on $[\TD,\TS]$ except the one for~$q(r)$,
it follows by a continuity argument that
$q(r(\TS))=\delta^\frac 32$.

\bigskip

\emph{Opposite sign case.}
Here we use both $R^+$ and $R^-$. First notice that on $[\TD,\TS]$,
\begin{equation} \label{mai3}
 R^- (1+ O(C \delta^2)) \le \frac{1}{q(r)} \le R^+ (1+ O(C \delta^2)). 
\end{equation}
From \eqref{eq:Rp},
\begin{equation*}
\frac{\ud }{\ud t} R^+\leq \frac {g_0}{\alpha}\left( 1+\frac{O(K)}{|\log \delta|}\right) (1+O(C\delta^2))
\leq \frac {g_0}{\alpha}\left(1+\frac{O(K)}{|\log \delta|}\right),
\end{equation*}
for $\delta>0$ small enough with respect to $C$.
Thus, for any $t\in[\TD,\TS)$, it holds by integration on $[\TD,t]$
\begin{equation}\label{bnf12}
R^+(t) \leq R^+(\TD) +\frac {g_0}{\alpha} \left(1+\frac{O(K)}{|\log \delta|}\right) (t-\TD).
\end{equation}
Similarly, we check that
\begin{equation}\label{bnf11}
R^-(t) \geq R^-(\TD) + \frac {g_0}{\alpha} \left(1-\frac{O(K)}{|\log \delta|}\right) (t-\TD).
\end{equation}
Note that \eqref{at:T1}, \eqref{mai3} and \eqref{bnf11} imply that
\[ \forall t \in [\TD,\TS), \quad  \frac 1{q(r(t))} \geq \frac 12 \left(\delta^{-2} + \frac{g_0}{\alpha} (t-\TD)\right). \]
In particular, the estimate on $q(r)$ in \eqref{BS:1} is strictly improved for $\delta$ small enough,
As a consequence $\TS=\infty$. 

Last, we observe that \eqref{mai3},  \eqref{bnf12} and \eqref{bnf11} imply~\eqref{eq:afaire2} where $\tilde C$ can be written in terms of $R^\pm(T_\delta)$, $T_\delta$ and $\delta$.
 \end{proof}

\subsection{Proof of Theorem~\ref{th:1}}
Let $\vec u$ be a $2$-solitary wave of~\eqref{nlkg} with $\sigma_{1}=\sigma_{2}$.

Let $\delta >0$ to be fixed later, and let $\TD$ and $\TS$ be as Proposition \ref{le:BS}. Using~\eqref{eq:E:R} at $t=\TS$ and~\eqref{eq:pr:th:1}, we have
\begin{equation*}
E(u(\TS),\partial_{t}u(\TS))=2E(Q,0)-c_1 g_0\delta^\frac 32 +O(|\log \delta|^{-1} \delta^\frac 32).
\end{equation*}
This allows to fix a $\delta >0$ small enough so that
\[ E(u(\TS),\partial_{t}u(\TS)) < 2 E(Q,0). \]
This contradicts the fact that the energy is decreasing and converges to $2E(Q,0)$ as
$t\to \infty$.

\subsection{Proof of Theorem~\ref{th:2}}
Let $\vec u$ be a $2$-solitary wave of~\eqref{nlkg} with $\sigma_{1}=-\sigma_{2}$.

\begin{proposition}
There exists $T >0 $ such that the decomposition of $\vec u$ satisfies
\begin{equation} \label{eq:strong}
\forall t \ge T, \quad  q(r(t))\lesssim t^{-1},\quad \mathcal N(t) \lesssim t^{-1}.
\end{equation}
\end{proposition}

\begin{proof}
Let $0 < \delta < \delta_1$ in the framework of Proposition \ref{le:BS}. From \eqref{eq:afaire2}, there exists $T>0$ large enough  such that
\begin{equation*}
\forall t \ge T/2, \quad \mathcal N(t) \leq C\delta,\quad q(r(t)) \leq \frac{2 g_0}{\alpha t}.
\end{equation*}
In particular,  from~\eqref{eq:b} and~\eqref{eq:damped}
\begin{equation}\label{bnf20}
\left| \frac{\ud b}{\ud t} - 2\nu^+ b\right|\lesssim \mathcal N^3 + t^{-1}\mathcal N,\quad
\frac{\ud}{\ud t}\mathcal F+2\mu \mathcal F\lesssim\mathcal N^3 + t^{-1}\mathcal N.
\end{equation}
Our goal is to obtain the decay rate of $\mathcal N$. The above bounds are not quite enough for this, due to the cubic term in $\mathcal N$ for which we have no decay yet, only smallness. In order to get around this, we will work on a modification $\tilde b$ of $b$.
Recall~\eqref{eq:new2}: 
\[ 0 \le \mathcal N^2 \lesssim \mathcal F+\frac{b}{2\mu}. \]
For $0<\epsilon \ll 1$ to be fixed, we observe that
\begin{equation}\label{bnf21}
\mathcal N^3 + t^{-1} \mathcal N \lesssim 
\epsilon^2 \mathcal N^2+ \epsilon^{-2} t^{-2}\lesssim
\epsilon^2\left(\mathcal F+\frac{b}{2\mu} \right)+  \epsilon^{-2}t^{-2}.
\end{equation}
(Here and below the implied constants do not depend on $\e$). Set $\tilde b=b- \epsilon  \mathcal F$ and observe that
\begin{equation*}
\tilde b=b- \epsilon \mathcal F
=\left(1+\frac{\epsilon}{2\mu}\right) b - \epsilon\left(\mathcal F+\frac{1}{2\mu} b\right)
\leq \left(1+\frac {\epsilon} {2\mu}\right) b  .
\end{equation*}
Therefore, using~\eqref{bnf20},
\begin{align*}
\frac{\ud \tilde b}{\ud t}
&\geq 2\nu^+ b + 2\epsilon \mu \mathcal F - C \epsilon^2\left(\mathcal F+\frac{b}{2\mu}\right)
-C\epsilon^{-2} t^{-2}\\
&\geq  (2\nu^+-\epsilon) b
+\left( 2 \epsilon \mu - C \epsilon^2 \right)\left(\mathcal F+\frac{b}{2\mu}\right)
-C\epsilon^{-2} t^{-2}\\
& \geq  \nu^+ \tilde b -C\epsilon^{-2} t^{-2},\end{align*}
where $\epsilon>0$ is taken small enough so that 
\begin{equation*}
(2\nu^+-\epsilon)\left(1+\frac {\epsilon} {2\mu}\right)^{-1}\geq \nu^+,\quad
2 \epsilon \mu - C \epsilon^2 >0.
\end{equation*}
This bound is suitable for our purpose. Integrating on $[t,\tau] \subset [T/2,+\infty)$, we obtain
\begin{equation*}
\tilde b(t) - e^{-\nu^+ (\tau-t)} \tilde b(\tau)
\lesssim \epsilon^{-2} \int_{t}^{\tau} e^{-\nu^+(s-t)} s^{-2} \ud s
\lesssim\epsilon^{-2}  t ^{-2}.
\end{equation*}
Passing to the limit as $\tau \to \infty$, one obtains for all $t\geq T/2$,
\begin{equation*}
 \tilde b(t) \lesssim \epsilon^{-2} t^{-2},\quad 
 b(t) \leq C \epsilon^{-2} t^{-2}+\epsilon \mathcal F(t).
\end{equation*}

Inserting this information into the equations \eqref{bnf20} of $\mathcal F$ and $b$  and using~\eqref{bnf21}, it holds
\begin{align*}
\frac{\ud}{\ud t} \left(\mathcal F+\frac{b}{2\mu}\right)
&\leq - 2\mu \mathcal F + \frac{\nu^+}{\mu} b + C\epsilon^2 \left(\mathcal F+\frac{b}{2\mu}\right)+ C\epsilon^{-2}t^{-2}\\
&\leq - (2\mu-C\epsilon^2) \left(\mathcal F+\frac{b}{2\mu}\right) + \left(1+\frac{\nu^+}{\mu} \right) b + C\epsilon^{-2}t^{-2}
\\
&\leq - \mu \left(\mathcal F+\frac{b}{2\mu}\right) + C\epsilon^{-2}t^{-2},
\end{align*}
by possibly choosing a smaller $\epsilon>0$.
Integrating on $[T/2,t]$, we obtain
\begin{equation*}
\left(\mathcal F+\frac{b}{2\mu}\right)(t)
- e^{- \mu (t-\frac{T}{2})} \left(\mathcal F+\frac{b}{2\mu}\right)(T/2)
\lesssim \int_{\frac{T}{2}}^t e^{- \mu (t-s)} s^{-2}\ud s
\end{equation*}
and so
\begin{equation*}
\mathcal N^2 (t)\lesssim \left(\mathcal F+\frac{b}{2\mu}\right)(t)
\lesssim   t^{-2} +  \delta^2 e^{\mu T/2} e^{-\mu t}
\end{equation*}
For $t \ge T>0$, the last term is bounded by $e^{-\mu t/2} \lesssim t^{-2}$, and this proves \eqref{eq:strong}.
\end{proof}

We complete the proof of Theorem~\ref{th:2}.
By~\eqref{def:ee} and~\eqref{eq:strong}, estimate~\eqref{eq:th:2} is proved.
Since~\eqref{eq:afaire2} holds for any $\delta>0$, it means  that
\begin{equation*}
\lim_{t\to \infty} \frac{1}{t q(r(t))} = \frac{g_0}{\alpha},
\end{equation*}
and, using the expansion~\eqref{Qdec} of $q$, this is equivalent to
\[ \lim_{t\to \infty} \frac{1}{t}  r(t)^{\frac{N-1}{2}} e^{r(t)} = \frac{\kappa g_0}{\alpha}. \]
This implies
\begin{equation*}
\lim_{t\to \infty} \left\{r(t)- \left(\log t-\frac 12(N-1)\log\log t+c_0\right)\right\}= 0,
\end{equation*}
for $c_0=\log \frac{\kappa g_0 }{\alpha} $. As $\mathcal N(t) \to 0$ (due to \eqref{eq:strong}) and in view of \eqref{eq:new1}, we obtain \eqref{eqq:th:2}.

Finally, we prove~\eqref{eqqq:th:2}. Let $1<\theta<\min(p-1,2)$.
We observe from~\eqref{for:y}, $|y-z|\lesssim \mathcal N$ and~\eqref{eq:strong} that
\begin{equation*}
\dot y = \frac 1{\alpha} \frac{z}{|z|} g(|z|)+O(\mathcal N^2+e^{-\theta |z|})
=\frac1{\alpha} \frac{y}{|z|} g(|z|)+O(t^{-\theta}).
\end{equation*}
Set $y=r\omega$, where $\omega\in \mathcal S_{\RR^N}(1)$.
We have $\dot y = \dot r \omega + r\dot \omega$, and using $\dot \omega \cdot \omega=0$, we find the estimate
$r |\dot \omega| = O(t^{-\theta})$. Thus, there exists $\omega^\infty\in \mathcal S_{\RR^N}(1)$ such that $\lim_{t\to \infty} \omega=\omega^\infty$.
Since $\displaystyle \frac{z}{\log t} = \frac{y}{\log t} +O(\mathcal N)$, we have proved $\displaystyle \lim_{t\to \infty} \frac{z(t)}{\log t}=\omega^\infty$.

From~\eqref{eq:z}, \eqref{eq:lbis} and~\eqref{eq:strong}, we also observe that
\begin{equation*}
 |2\dot z_1- { \dot z} |+|2 \dot z_2+{\dot z}| \lesssim t^{-\theta},
\end{equation*}
which is sufficient to complete the proof of~\eqref{eqqq:th:2}.

\section{Proof of Theorem~\ref{th:3}}\label{S:4}

\subsection{Preliminary result}
We use the notation from the beginning of Section~\ref{S:2}. We also use the constant
\begin{equation}
\beta := \frac{1}{2 \sqrt{\alpha^2 + \nu_0^2}} = \langle \vec Y^+, \vec Z^+ \rangle^{-1} >0.
\end{equation}

\begin{lemma}\label{le:W}
For any $(z_{1},z_{2},\ell_{1},\ell_{2})\in \RR^{4N}$ with $|z|$ large enough, there exist linear maps
\begin{equation*}
B:\RR^2 \to \RR^2,\quad  V_j:\RR^2\to \RR^2 \quad \text{for } j=1,\ldots,N,
\end{equation*}
smooth in $(z_{1},z_{2},\ell_{1},\ell_{2})$, satisfying
\begin{equation*}
\|B- \beta \Id\|\lesssim e^{-\frac 12|z|}, \quad \|V_j\|\lesssim e^{-\frac 12 |z|},
\end{equation*}
and such  that the function $W(a_1,a_2):\RR^N\to \RR$ defined by
\begin{equation*}
W(a_1,a_2) := \sum_{k=1,2} \bigg\{ B_k(a_1,a_2) Y_k+ \sum_{j=1}^N V_{k,j} (a_1,a_2)\partial_{x_j} Q_k\bigg\},
\end{equation*}
satisfies, for all $k=1,2$, $j=1,\ldots,N$,
\begin{equation*}
\langle W(a_1,a_2),\partial_{x_j} Q_k\rangle =0,\quad \langle W(a_1,a_2),Y_k\rangle = \beta a_k.
\end{equation*}
In particular, setting
\begin{equation*}
\vec W(a_1,a_2)=\begin{pmatrix} W(a_1,a_2)\\ \nu^+ W(a_1,a_2)\end{pmatrix}\quad\text{it holds}\quad
\langle \vec W(a_1,a_2),\vec Z_k^+\rangle= a_k.
\end{equation*}
\end{lemma}

\begin{proof} 
Define
\begin{equation*}
W(a_1,a_2)=\sum_{k=1,2} \bigg\{ b_k Y_k+ \sum_{j=1}^N v_{k,j} \partial_{x_j} Q_k\bigg\},
\end{equation*}
our goal is to solve for $b_k$, $v_{k,j}$ in function of $a_1, a_2$.
Using the relations $\langle Y_1,\partial_{x_j} Q_1\rangle =0$, $\langle \partial_{x_{j'}} Q_1,\partial_{x_j} Q_1\rangle =0$ for $j\neq j'$, and the estimate $\langle \partial_{x_{j'}} Q_2,\partial_{x_j} Q_1\rangle =O(e^{-\frac 12 |z|})$ for any $j, j'$
(see~\eqref{tech1}), we observe that the condition $\langle W(a_1,a_2),\partial_{x_j} Q_1\rangle=0$ is equivalent to
a linear relation between the coefficients in the definition of $W(a_1,a_2)$ of the form
\begin{equation*}
 \|\partial_{x_j} Q\|_{L^2}^2  v_{1,j} =
O(e^{-\frac 12 |z|}) b_2 + \sum_{j'=1}^N  O(e^{-\frac 12 |z|})v_{2,j'} .
\end{equation*}
Similarly, the condition $\langle W(a_1,a_2),\partial_{x_j} Q_2\rangle=0$ is equivalent to
\begin{equation*}
\|\partial_{x_j} Q\|_{L^2}^2 v_{2,j}  =
O(e^{-\frac 12 |z|}) b_1 + \sum_{j'=1}^N  O(e^{-\frac 12 |z|})v_{1,j'} .
\end{equation*}
Moreover, since $\langle Y,Y\rangle=1$ (see Lemma~\ref{le:L}) and $\langle Y_1,Y_2\rangle=O(e^{-\frac 12|z|})$, the conditions $\langle W(a_1,a_2),Y_k\rangle=\beta a_k$  write
\begin{align*}
b_1 &=  \beta a_1 + O\bigg\{\bigg(|b_2|+\sum_{j'=1}^N |v_{2,j'}|\bigg) e^{-\frac 12 |z|}\bigg\},\\
b_2 &=  \beta a_2 + O\bigg\{\bigg(|b_1|+\sum_{j'=1}^N |v_{1,j'}|\bigg) e^{-\frac 12 |z|}\bigg\}.
\end{align*}
The existence and desired properties of $b_k$ and $v_{k,j}$ for $k=1,2$ and $j=1,\ldots,N$ thus follow from inverting a linear system for $|z|$ large enough.
\end{proof}

\subsection{Construction of a family of 2-solitary waves}
\begin{proposition}\label{pr:exist}
Assume $\sigma=-1$. For  $\delta>0$ small enough, for any
\begin{equation}\label{eq:pr:10}\left\{\begin{aligned}
& (z_1(0),z_2(0))\in \RR^{2N}   \text{ with } |z_1(0)-z_2(0)|> 5 |\log \delta|, \\
&( \ell_1(0),\ell_2(0))\in \mathcal B_{\RR^{2N}}(\delta),\\
& \vvep(0)\in \mathcal B_\ENE(\delta) \text{ with } \eqref{ortho} \text{ and } \langle \vvep(0),\vec Z_k^+(0)\rangle=0\ \ \text{for}\ k=1,2,
\end{aligned}\right.\end{equation}
there exists $(a_1^+(0),a_2^+(0))\in  \bar{\mathcal B}_{\RR^2}(\delta^\frac 54)$ such that the solution $\vec{u}$ of~\eqref{nlkg} with the  initial data
\begin{equation*}
\vec u(0)= \vec Q_1(0) + \vec Q_2(0) + \vec W(a_1^+(0),a_2^+(0)) + \vvep(0)
\end{equation*}
is a $2$-solitary wave.
\end{proposition}

\begin{proof}
\emph{Decomposition.}
For any $t\geq 0$ such that the solution $\vec u(t)$ satisfies~\eqref{for:dec}, we decompose it according to Lemma~\ref{le:dec}.
Note that by the properties of the function $W$ in Lemma~\ref{le:W} and the orthogonality properties \eqref{ortho} of  $\vvep(0)$ assumed in~\eqref{eq:pr:10}, the initial data $\vec u(0)$ is modulated, in the sense that $ (z_1(0),z_2(0),\ell_1(0),\ell_2(0))$ and
\begin{equation*}
\vve(0)=\vec W(a_1^+(0),a_2^+(0)) + \vvep(0),
\end{equation*}
are the parameters of the decomposition of $\vec u(0)$.
 In particular, it holds from~\eqref{eq:pr:10}
\begin{equation}\label{at:0}
\mathcal N(0)\lesssim \delta,\quad q(r(0))\lesssim \delta^2.
\end{equation}
Moreover, by Lemma~\ref{le:W}, for $k=1,2$,
\begin{equation*}
\langle \vve(0),\vec Z_k^+(0)\rangle =  \langle \vec W(a_1^+(0),a_2^+(0)), Z_k^+ \rangle =  a_k^+(0),
\end{equation*}
which is consistent with the definition of $a_k^+$ in (v) of Lemma~\ref{le:dec}.

\bigskip

\emph{Bootstrap estimates.}
We introduce the following bootstrap estimates
\begin{equation}\label{BS:2}
\mathcal N\leq \delta^{\frac 34},\quad
q(r)\leq \delta^{\frac 32},\quad
b\leq  \delta^{\frac 52}.
\end{equation}
We set
\begin{equation*}
\TS=\sup\left\{ t\in [0,\infty)\hbox{ such that \eqref{BS:2} holds on $[0,t]$}\right\}\geq 0.
\end{equation*}

\bigskip

\emph{Estimates on the damped components.}
The estimate on $\mathcal N$ is strictly improved on $[0,\TS]$ as in the proof of Proposition~\ref{le:BS}.
In particular, $\mathcal N\lesssim \delta$ on $[0,\TS]$.

\bigskip

\emph{Estimate on the distance.}
Note that $r(t)\gtrsim \log \delta$ and thus, from \eqref{eq:Rp},
\begin{equation*}
\frac{\ud R^-}{\ud t} \geq \frac  {g_0}{\alpha}\left( 1-\frac{C}{|\log \delta|}\right) (1-C\delta^2)
\geq \frac {g_0}{2\alpha}.
\end{equation*}
Thus, by integration on $[0,t]$, for any $t\in [0,\TS]$, it holds
$R^-(t) \geq R^-(0) + \frac{g_0 t}{2\alpha}$.
Using \eqref{mai3} and next \eqref{at:0},  we obtain
\begin{equation*}
\frac 1{q(r(t))} \geq \left(\frac {1-C \delta^2}{q(r(0))}+ \frac{g_0 t}{2\alpha}\right)\big(1-C\delta^{2}\big) \gtrsim \delta^{-2} .
\end{equation*}
This strictly improves the estimate of $q(r)$ in \eqref{BS:1}.

\bigskip

\emph{Transversality condition.}
From \eqref{eq:b} and $\mathcal N\lesssim \delta$, we observe that for any time $t\in [0,\TS]$ where it holds $b(t)=\delta^{\frac 52}$, we have
\begin{equation*}
\frac \ud{\ud t} b(t) \geq 2\nu^+ b(t) - C \delta^{3}
\geq 2\nu^+  \delta^{\frac 52} - C \delta^{3} \geq \nu^+  \delta^{\frac 52} >0,
\end{equation*}
for $\delta>0$ small enough. This transversality condition is enough to justify the existence of at least a couple $(a_1^+(0),a_2^+(0))\in 
\bar{\mathcal B}_{\RR^2}( \delta^\frac 54)$ such that $\TS=\infty$.

Indeed, for the sake of contradiction assume that for all $(a_1^+(0),a_2^+(0))\in \bar{\mathcal B}_{\RR^2}( \delta^\frac 54)$, it holds $\TS<\infty$.
Then, a contradiction follows from the following observations (see for instance more details in \cite{CMMgn} or in~\cite[Section 3.1]{CMkg}).

\begin{description}
\item[Continuity of $\TS$] the above transversality condition proves that the map 
\begin{equation*}
(a_1^+(0),a_2^+(0))\in \bar{\mathcal B}_{\RR^2}(\delta^\frac 54)
\mapsto \TS\in [0,+\infty)
\end{equation*}
is continuous and that $\TS=0$ when $(a_1^+(0),a_2^+(0))\in \mathcal S_{\RR^2}(\delta^\frac 54)$.
\item[Construction of a retraction] As a consequence, the map
\begin{equation*}
(a_1^+(0),a_2^+(0))\in \bar{\mathcal B}_{\RR^2}( \delta^\frac 54)
\mapsto (a_1^+(\TS),a_2^+(\TS))\in \mathcal S_{\RR^2}( \delta^\frac 54)
\end{equation*}
is continuous and its restriction to the  sphere $\mathcal S_{\RR^2}( \delta^\frac 54)$ is the identity.
\end{description}
This is a contradiction with the no retraction theorem for continuous maps from the ball to the sphere.
\end{proof}

\begin{remark}
The use of initial data similar to the ones in Proposition~\ref{pr:exist} allows to correct a flaw in the articles \cite{CMMgn,CMkg,MMwave} 
dealing with the construction of multi-solitons in several contexts.
For example, in \cite{CMkg}, the initial data $U_{0}$ chosen page 18 is not exactly modulated, and this is why Lemma~6 provides the modulation keeping track of the free parameters necessary for the topological argument.
However, this modulation involves a translation parameter denoted by $\tilde{\mathbf{y}}$ in \cite{CMkg} (similar to $(z_{1},z_{2})$ in the present paper) on which the topological argument finally depends.
To  close the choice of the free parameters for  initial data as chosen in~\cite{CMkg}, an extra argument would be needed (like a fixed point result). It is simpler though to define an initial data already modulated as in Proposition~\ref{pr:exist}, or in \cite[Lemma 3.1]{JJ}.
\end{remark}

\subsection{Lipschitz estimate and uniqueness}

The heart of Theorem \ref{th:3} is the following proposition.

\begin{proposition}\label{pr:lip}
For $\delta>0$ small enough, let 
$(z_1(0),z_2(0),\ell_1(0),\ell_2(0),\vvep(0))$ and
$(\tilde z_1(0),\tilde z_2(0),\tilde \ell_1(0),\tilde \ell_2(0),\tvvep(0))$
be as in \eqref{eq:pr:10},
and let $(a_1^+(0),a_2^+(0))\in   {\mathcal B}_{\RR^2}(\delta)$, $(\tilde a_1^+(0),\tilde a_2^+(0))\in   {\mathcal B}_{\RR^2}(\delta)$
be such that the solutions $\vec u(t)$ and $\vec {\tilde u}(t)$ of \eqref{nlkg} 
with initial data
\begin{align*}
& \vec u(0)= \vec Q_1(0) + \vec Q_2(0) + \vec W(a_1^+(0),a_2^+(0)) + \vvep(0),\\
& \vec {\tilde u}(0)= \vec {\tilde Q}_1(0) + \vec {\tilde Q}_2(0) + \vec W(\tilde a_1^+(0),\tilde a_2^+(0)) + \tvvep(0),
\end{align*}
are $2$-solitary waves.
Then it holds 
\begin{align*}
&\sum_{k=1,2}|a_{k}^{+}(0)-{\tilde{a}}_{k}^{+}(0)| 
\\&\qquad \lesssim \delta^{\frac 14} \bigg(\|\vvep(0)-\tvvep(0)\|_\ENE+
\sum_{k=1,2} \left\{|z_{k}(0)-\tilde{z}_{k}(0) |+|\ell_{k}(0)-\tilde{\ell}_{k}(0)|\right\}
\bigg).
\end{align*}
\end{proposition}

\begin{proof}
We will split the proof in several steps.
For $\delta>0$ small enough, from Proposition~\ref{le:BS} and the assumption on their initial data $\vec u(0)$, $\vec {\tilde u}(0)$,  the $2$-solitary waves $\vec u$ and $\vec{\tilde u}$ decompose for any $t\geq 0$ as in Lemma~\ref{le:dec}
$\vec{u}=\vec{Q}_1 +\vec{Q}_2 +\vve$, $\vec{\tilde{u}}=\vec{\tilde{Q}}_1 +\vec{\tilde{Q}}_2 +\vec{\tilde\e}$
and the parameters of their decompositions $(z_k,\ell_k, \mathcal N)_{k=1,2}$ and $(\tilde z_k,\tilde \ell_k, \tilde{\mathcal N})_{k=1,2}$ respectively satisfy~\eqref{eq:afaire2} for all $t\geq 0$.
Denote
\begin{gather}\label{eq:ch}
\check{z}_{k}=z_{k}-\tilde{z}_{k},\quad \cbk=\ell_{k}-{\tilde\ell}_{k},\quad \check{Q}_{k}=Q_{k}-{\tilde Q}_{k},\quad
\cce=\e-\tilde\e,\quad \check{\eta}=\eta-\tilde\eta,\\
\check{\mathcal N}=\bigg[\|\vec{\cv}\|_\ENE^2+\sum_{k=1,2} |\check{z}_{k} |^2+\sum_{k=1,2} |\check{\ell}_{k} |^2\bigg]^{\frac 12}.
\end{gather}
(Notice the extra term in $\check{\mathcal N}$).

If $\sum_{k=1,2} |\check z_k(0)|> \delta^\frac 34$, since $(a_1^+(0),a_2^+(0))$, $(\tilde a_1^+(0),\tilde a_2^+(0))\in  \bar{\mathcal B}_{\RR^2}(\delta)$, the desired estimate is satisfied. Thus, we assume now 
 $\sum_{k=1,2} |\check z_k(0)|\leq \delta^\frac 34$, and we work only for time $t$ such that $\sum_{k=1,2} |\check z_k(t)|$
 is small (see~\eqref{BS:lip}).

\bigskip

\emph{Equation of $\vec \cv$.}
By direct computation from \eqref{syst_e} for $\vec\e$ and $\vec{\tilde \e}$, we have
\begin{equation}\label{eq:ce}
\left\{\begin{aligned}
\partial_t \cce&=\check{\eta}
+\md_{\cv,{1}}+\md_{\cv,{2}}\\
\partial_t \check{\eta}&=\Delta \cce-\cce+f'(R)\cce-2\alpha\check\eta+\md_{\ce,{1}}+\md_{\ce,{2}}+\check{G}
\end{aligned}\right.\end{equation}
where
\begin{align*}
 \md_{\cv,{1}}&=\sum_{k=1,2}\left(\dot{\check{z}}_{k}-{\check{\ell}_{k}}\right)\cdot\nabla Q_{k},\quad \md_{\cv,{2}}=\sum_{k=1,2}\left(\dot{\tilde{z}}_{k}-{\tilde\ell_{k}}\right)\cdot\nabla \check{Q}_{k},\\
 \md_{\ce,{1}}&=\sum_{k=1,2}\left(\dot{\check{\ell}}_{k}+2\alpha {\check{\ell}}_{k}\right)\cdot\nabla Q_{k}\\
&\quad -\sum_{k=1,2}({\check{\ell}}_{k}\cdot\nabla)(\dot{ {z}}_{k}\cdot\nabla) Q_{k}
-\sum_{k=1,2}({\tilde{\ell}}_{k}\cdot\nabla)(\dot{\check{z}}_{k}\cdot\nabla) Q_{k},\\
  \md_{\ce,{2}}&=\sum_{k=1,2}\left(\dot{\tilde{\ell}}_{k}+2\alpha {\tilde\ell}_{k}\right)\cdot\nabla \check{Q}_{k}-\sum_{k=1,2}({\tilde\ell}_{k}\cdot\nabla)(\dot{\tilde z}_{k}\cdot\nabla)\check{Q}_{k}
\end{align*}
and
\begin{equation*}
\check{G}=\big[f(R+\e)-f(R)-f'(R)\e\big]
-\big[f({\tilde{R}}+\tilde{\e})-f({\tilde{R}})-f'(R)\tilde{\e}\big]+G-\tilde G.
\end{equation*}

{\bf Step 1.}  We claim the following estimate on the nonlinear term $\check G$.

\begin{lemma}
\begin{equation}\label{on:checkG}
\|\check G\|_{L^2}\lesssim \check {\mathcal N} (\mathcal N+\tilde {\mathcal N}+e^{-\frac 12 |z|}).
\end{equation}
\end{lemma}

\begin{proof}
We will in fact derive pointwise bounds. We decompose $\check{G}=\sum_{j=1}^5\check{G}_j$ where
\begin{align*}
&\check{G}_1 =[f'(R)-f'(\tilde R)] \tilde \e,\\
&\check{G}_2 = [f(R+\e)-f(R)-f'(R)\e]- [f(R+\tilde{\e})-f(R)-f'(R)\tilde{\e}],\\
&\check{G}_3 = [f(R+\tilde\e)-f(R)-f'(R)\tilde\e]- [f(\tilde{R}+\tilde{\e})-f(\tilde{R})-f'(\tilde{R})\tilde{\e}],\\
&\check{G}_4 = [f(Q_1 +Q_2 )-f(Q_2 )]-[f(Q_1 +\tilde{Q}_2 )-f(\tilde{Q}_2 )],\\
&\check{G}_5 = [f(Q_1 +\tilde{Q}_2 )-f(Q_1 )]-[f(\tilde{Q}_1 +\tilde{Q}_2 )-f(\tilde{Q}_1)].
\end{align*}
For $\check{G}_1$, using $|f'(R)-f'(\tilde{R}) |\lesssim  \left(|Q_1 |^{p-1}+|Q_2|^{p-1}\right)\sum_{k=1,2}|\czk|$, we have
\begin{equation*}
|\check{G}_1|\lesssim |\tilde \e| \left(|Q_1 |^{p-1}+|Q_2|^{p-1}\right)\sum_{k=1,2}|\czk|.
\end{equation*}
For $\check G_2$, by Taylor formular, we have
\begin{align*}
\check G_2 & = f(R+\tilde \e + \check \e) - f(R+\tilde{\e}) - f'(R) \check{\e} \\
& = \check{\e}  \int_0^1 f'(R+ \tilde \e + \theta \cv) - f'(R) \ud \theta  
= \check{\e} \int_0^1 ( \tilde \e + \theta  \cv) \int_0^1  f''(R+ \rho (\tilde{\e}+\theta \check{ \e})) \ud \rho \ud \theta \\
& = \check{\e}   \int_0^1 \int_0^1 ( \theta \e + (1-\theta) \tilde \e)  f''(R+ \rho( \tilde{\e}+\theta \check{ \e})) \ud \rho \ud \theta.
\end{align*}
As $|f''(x+y)| = p(p-1) |x+y|^{p-2} \lesssim |x|^{p-2} + |y|^{p-2}$, we get
\[ |\check G_2| \lesssim |\check{\e}| (|\e| + |\tilde \e |) \left( |R|^{p-2} + |\tilde \e|^{p-2}+|\e|^{p-2} \right). \]

For $\check G_3$: by Taylor formula,
\begin{align*}
\MoveEqLeft \big|[f(R+\tilde\e)-f(R)]-[f(\tilde{R}+\tilde\e)-f(\tilde{R})]\big| =\Big|\tilde\e\int_{0}^{1}[f'(R+\theta\tilde\e)-f' (\tilde{R}+\theta\tilde\e )]\ud\theta\Big| \\
& \lesssim (|\check z_{1}|+|\check z_{2}|)|\tilde\e|\big(|R|^{p-2}+|\tilde{R}|^{p-2}+|\tilde\e|^{p-2}\big),
\end{align*}
and
\begin{equation*}
\big| [f'(R)-f'(\tilde{R})]\tilde\e\big|\lesssim(|\check z_{1}|+|\check z_{2}|)|\tilde\e|\big(|R|^{p-2}+|\tilde{R}|^{p-2}\big).
\end{equation*}
It follows that
\begin{equation*}
|\check G_3| \lesssim (|\check z_{1}|+|\check z_{2}|)|\tilde\e|\big(|R|^{p-2}+|\tilde{R}|^{p-2}+|\tilde\e|^{p-2}\big).
\end{equation*}

For $\check G_4$: Taylor formula gives similarly
\begin{align*}
 |\check{G}_4  |&=\left|Q_1 \int_{0}^{1}\big[f' (Q_2 +\theta Q_1  )-f' (\tilde{Q}_2 +\theta Q_1  )\big]\ud\theta\right|\\
&\lesssim |Q_1 | |Q_2 -\tilde{Q}_2  |\left(|Q_1 |^{p-2}+|Q_2 |^{p-2}\right)\lesssim |\check{z}_2 |\left(|Q_1 |^{p-1}|Q_2 |+|Q_1 ||Q_2 |^{p-1}\right).
\end{align*}

For $\check G_5$, a similar argument yields
\[ |\check{G}_5|\lesssim |\check{z}_1 | (|Q_1 |^{p-1}|Q_2 |+|Q_1 ||Q_2 |^{p-1} ). \]

Estimate \eqref{on:checkG} then follows from \eqref{tech2} and Sobolev embeddings.
\end{proof}

{\bf Step 2.}  Equations for the parameters.

\begin{lemma} \label{le:chP} 
{}\
\begin{enumerate}
\item \emph{Equations of the geometric parameters.} For $k=1,2$,
\begin{align}
 |\dot{\check{z}}_{k}-{\check{\ell}_{k}} |  
&\lesssim  \check{\mathcal N}  (\mathcal N+\tilde {\mathcal N} ) ,\label{eq:cz}\\
 |\dot {\check \ell}_{k} +2\alpha \check\ell_{k}| 
&\lesssim  \check{\mathcal N}  (\mathcal N+\tilde {\mathcal N}+e^{-\frac 12 |z|}).
\label{eq:cl}
\end{align}
\item \emph{Equations of the exponential directions.}
Let $\check{a}_{k}^{\pm} = \langle\vec{\cce},\vec{Z}^{\pm}_{k}\rangle$. Then,
\begin{align}\label{cz-a}
\bigg| \frac{\ud }{\ud t} \check{a}_{k}^{\pm} - \nu^\pm \check{a}_{k}^{\pm}\bigg|
\lesssim \check{\mathcal N}  (\mathcal N+\tilde {\mathcal N}+e^{-\frac 12 |z|}) .
\end{align}
\end{enumerate}
\end{lemma}

\begin{proof} (i)
From the orthogonality conditions in~\eqref{ortho}, we derive the following relations: for all $k=1,2$ and $j=1,\ldots,N$,
\begin{equation}\label{c:ortho}
\langle\cce, \partial_{x_{j}} Q_{k}\rangle+\langle\tilde \e, \partial_{x_{j}} \check{Q}_{k}\rangle= \langle\check{\eta}, \partial_{x_{j}} Q_{k}\rangle
+\langle\tilde\eta, \partial_{x_{j}} \check{Q}_{k}\rangle=0.
\end{equation}
We derive  the equations of $\czk$ and $\cbk$ from these relations.

First, from~\eqref{c:ortho} for $\check \e$ with $k=1$ and $j=1,\ldots,N$,
\begin{align*}
0&= \frac \ud{\ud t}\left(\langle\cce, \partial_{x_{j}} Q_1 \rangle+\langle\tilde \e, \partial_{x_{j}} \check{Q}_1 \rangle\right)
\\& =\langle\partial_{t }\cv,\partial_{x_{j}} Q_1 \rangle
+\langle\cv,\partial_{t} \partial_{x_{j}} Q_1\rangle+\langle\partial_{t}\tilde\e,\partial_{x_{j}} \check{Q}_1 \rangle+\langle\tilde\e,\partial_{t} \partial_{x_{j}} \check{Q}_1\rangle.
\end{align*}
Thus, using~\eqref{eq:ce} for $\partial_{t} \cv$ and \eqref{syst_e} for $\partial_{t}\tilde \e$,
\begin{align*}
0&=\langle\check{\eta}, \partial_{x_{j}} Q_1\rangle+(\dot{\check{z}}_{1,j} -{\check{\ell}_{1,j} })
\|\partial_{x_{j}} Q_1\|_{L^2}^2+\langle (\dot{\check{z}}_2 -{\check{\ell}_2 })
\cdot \nabla Q_2 ,\partial_{x_{j}} Q_{1}\rangle\\
&\quad +\sum_{k=1,2}\langle (\dot{\tilde{z}}_{k}-{\tilde\ell_{k}})\cdot \nabla \check{Q}_{k},\partial_{x_{j}} Q_1\rangle
-\langle\cce,\left(\dot{z}_1 \cdot\nabla\right)\partial_{x_{j}} Q_1 \rangle\\
&\quad+\langle\tilde\eta,\partial_{x_{j}}  \check{Q}_1 \rangle
+\sum_{k=1,2}\langle  (\dot{\tilde z}_{k}-{\tilde\ell_{k}} )\cdot  \nabla \tilde{Q}_k ,\partial_{x_{j}} \check Q_1\rangle\\
&\quad-\langle \tilde\e, (\dot{\check{z}}_1 \cdot\nabla )\partial_{x_{j}}   Q_1\rangle
-\langle\tilde\e, (\dot{\tilde{z}}_1 \cdot\nabla )\partial_{x_{j}}   \check{Q}_1 \rangle.
\end{align*}
Using \eqref{tech1}, \eqref{eq:z}, \eqref{c:ortho} and the estimates
\begin{equation*}
|\langle \partial_{x_{j'}} \check Q_1,\partial_{x_{j}} Q_{1}\rangle|+
|\langle \partial_{x_{j'}} \check Q_2,\partial_{x_{j}} Q_{1}\rangle|+
|\langle \partial_{x_{j'}}  Q_2,\partial_{x_{j}} \check Q_{1}\rangle|\lesssim |\check z_{1}| +|\check z_{2}|,
\end{equation*}
it follows that
\begin{align*}
 |\dot{\check{z}}_{1,j} -{\check{\ell}_{1,j} } |
&\lesssim  e^{-\frac 12 |z|}|\dot{\check{z}}_2 -{\check{\ell}_2 } |
+(|\check z_{1}|+|\check z_2|)\tilde {\mathcal N}^2 + (|\ell_{1}|+\mathcal N^2) \check{\mathcal N}\\
&\quad  + \tilde {\mathcal N} (|\check \ell_{1}|+ |\dot{\check{z}}_{1} -{\check{\ell}_{1} } |) 
+\tilde {\mathcal N} |\check z_{1}| (\mathcal N^2+|\ell_{1}|).
\end{align*}
By symmetry, an analogous estimate holds for $ |\dot{\check{z}}_{2,j} -{\check{\ell}_{2,j} } |$. Summing up both gives \eqref{eq:cz}.

Second, using~\eqref{c:ortho} for $\check \eta$ with $k=1$ and $j=1,\ldots,N$, we have
\begin{align*}
0&= \frac \ud{\ud t}\left(\langle\check{\eta}, \partial_{x_{j}} Q_1 \rangle+\langle\tilde \eta, \partial_{x_{j}} \check{Q}_1 \rangle\right)
\\& =\langle\partial_{t }\check{\eta},\partial_{x_{j}} Q_1 \rangle
+\langle\check{\eta},\partial_{t} \partial_{x_{j}} Q_1\rangle+\langle\partial_{t}\tilde\eta,\partial_{x_{j}} \check{Q}_1 \rangle
+\langle\tilde\eta,\partial_{t} \partial_{x_{j}} \check{Q}_1\rangle.
\end{align*}
Using~\eqref{eq:ce}, one has
\begin{align*}
&\langle\partial_{t}\check{\eta},\partial_{x_{j}} Q_1\rangle
\\&\quad =\langle\Delta \cce-\cce+f'(Q_1 )\cce,\partial_{x_{j}} Q_1\rangle
-2\alpha\langle\check{\eta},\partial_{x_{j}} Q_1 \rangle\\
&\qquad+ \langle [f'(R)-f'(Q_1)]\cce,\partial_{x_{j}} Q_1 \rangle
+\langle\check{G},\partial_{x_{j}} Q_1  \rangle\\
&\qquad+ (\dot{\check{\ell}}_{1,j}+2\alpha {\check{\ell}}_{1,j}  ) \|\partial_{x_{j}} Q_1  \|_{L^{2}}^{2}
+ \langle(\dot{\check{\ell}}_2+2\alpha {\check{\ell}}_2  )\cdot \nabla Q_2 ,\partial_{x_{j}} Q_1 \rangle\\
&\qquad-\sum_{k=1,2}\langle({\check{\ell}}_{k}\cdot\nabla)(\dot{{z}}_{k}\cdot\nabla) Q_{k},\partial_{x_{j}} Q_1 \rangle
-\sum_{k=1,2}\langle({\tilde{\ell}}_{k}\cdot\nabla)(\dot{\check{z}}_{k}\cdot\nabla) Q_{k},\partial_{x_{j}} Q_1 \rangle\\
&\qquad
+\sum_{k=1,2}\langle (\dot{\tilde{\ell}}_{k}+2\alpha {\tilde\ell}_{k} )\cdot\nabla \check{Q}_{k},\partial_{x_{j}}  Q_1\rangle
-\sum_{k=1,2}\langle({\tilde\ell}_{k}\cdot\nabla)(\dot{\tilde z}_{k}\cdot\nabla)\check{Q}_{k},\partial_{x_{j}}  Q_1 \rangle.
\end{align*}
Next,
\begin{align*}
\langle\check{\eta},\partial_{t} \partial_{x_{j}} Q_1\rangle &=
-\langle\check{\eta},(\dot z_{1}\cdot\nabla) \partial_{x_{j}} Q_1\rangle\\
&=-\langle\check{\eta},((\dot z_{1}-\ell_{1})\cdot\nabla) \partial_{x_{j}} Q_1\rangle
-\langle\check{\eta},(\ell_{1}\cdot\nabla) \partial_{x_{j}} Q_1\rangle
=O( \check {\mathcal N} \mathcal N) .
\end{align*}
Then, using~\eqref{syst_e},
\begin{align*}
\langle \partial_{t}\tilde\eta, \partial_{x_{j}}  \check{Q}_1\rangle&
=\langle\Delta \tilde{\e}-\tilde{\e}+f'(\tilde{Q}_1 )\tilde{\e},\partial_{x_{j}}  \check{Q}_1 \rangle
-2\alpha\langle\tilde{\eta},\partial_{x_{j}}  \check{Q}_1\rangle\\
&\quad +\langle f({\tilde{R}}+\tilde\e)-f({\tilde{R}})-f'(\tilde Q_{1})\tilde \e,\partial_{x_{j}}  \check{Q}_1 \rangle\\
&\quad 
+\sum_{k=1,2} \langle (\dot{\tilde \ell}_{k}+2\alpha{\tilde\ell}_{k} )\cdot\nabla \tilde Q_{k},\partial_{x_{j}} \check{Q}_1 \rangle
\\
&\quad 
-\sum_{k=1,2}\langle({\tilde\ell}_{k}\cdot\nabla)(\dot{\tilde z}_{k}\cdot\nabla) \tilde Q_{k},\partial_{x_{j}} \check{Q}_1 \rangle+\langle\tilde{G},\partial_{x_{j}}  \check{Q}_1\rangle.
\end{align*}
Last,
\begin{align*}
\langle\tilde\eta,\partial_{t} \partial_{x_{j}} \check{Q}_1\rangle&=
-\langle \tilde\eta, (\dot{\check{z}}_1 \cdot\nabla )\partial_{x_{j}}   Q_1\rangle
-\langle\tilde\eta, (\dot{\tilde{z}}_1 \cdot\nabla )\partial_{x_{j}}   \check{Q}_1 \rangle\\
&=-\langle \tilde\eta, (\dot{\check{z}}_1-\check\ell_{1}) \cdot\nabla )\partial_{x_{j}}   Q_1\rangle
-\langle \tilde\eta,  \check \ell_{1}\cdot\nabla )\partial_{x_{j}}   Q_1\rangle\\
&\quad -\langle\tilde\eta, ((\dot{\tilde{z}}_1-\tilde \ell_{1}) \cdot\nabla )\partial_{x_{j}}   \check{Q}_1 \rangle
-\langle \check\eta,  \tilde \ell_{1}\cdot\nabla )\partial_{x_{j}}   Q_1\rangle\\
&=O(\tilde {\mathcal N} |\dot{\check{z}}_1-\check\ell_{1}|)+O( \check {\mathcal N} \tilde {\mathcal N}).
\end{align*}
By~\eqref{on:checkG}, we have
$ |\langle\check{G},\nabla Q_{k}\rangle |\lesssim  \check{\mathcal N}\big( \mathcal N+\tilde{\mathcal N}+e^{-\frac 12 |z|}\big).$

Combining these estimates with $\Delta \partial_{x_{j}} Q_1 - \partial_{x_{j}} Q_1+f'(Q_1 ) \partial_{x_{j}} Q_1=0$, \eqref{c:ortho} for $\check \eta$, we obtain
\begin{equation*}
|\dot{\check\ell}_{1,j}+2\alpha \check \ell_{1,j}|
\lesssim 
\bigg(\check{\mathcal N}+ \sum_{k=1,2}\left\{ |\dot {\check z}_{k}-\check\ell_{k}|+|\dot {\check\ell}_{k}+2\alpha \check\ell_{k}|\right\}\bigg)
\big(\tilde {\mathcal N}+\mathcal N + e^{-\frac 12|z|}\big).
\end{equation*}
We argue similarly to obtain the same bound on $|\dot{\check\ell}_{2,j}+2\alpha \check \ell_{2,j}|$. Summing up and using  \eqref{eq:cz}, we obtain \eqref{eq:cl}.

(ii) The proof of~\eqref{cz-a} is similar. See also the proof of (v) in Lemma~\ref{le:dec}.
\end{proof}

{\bf Step 3.} Energy estimates.
For $\mu>0$ small to be chosen, denote $\rho=2\alpha-\mu$, and consider the following energy functional
\begin{equation}
\check{\mathcal{E}}=\int  \big\{|\nabla \cce|^{2}+(1-\rho\mu)\cce^{2}+(\check{\eta}+\mu\cce)^{2}-f'(R)\cce^{2}\big\}.
\end{equation}

\begin{lemma}\label{le:chE}
There exists $\mu>0$ such that the following hold.
\begin{enumerate}
\item \emph{Coercivity and bound.}
\begin{equation}\label{coer:ce}
2\mu \|\vec \cv\, \|_{\ENE}^2- \frac 1{2\mu}   \bigg[\tilde{\mathcal N}^2 \check{\mathcal N}^2
+\sum_{k=1,2}\left((\check a_{k}^+)^2+(\check a_{k}^-)^2\right)\bigg]\leq
\check{\mathcal{E}} \leq \frac 1\mu \|\vec \cv\, \|_{\ENE}^2.
\end{equation}
\item \emph{Time variation.}
\begin{equation*}
\frac \ud{\ud t}\check{\mathcal{E}}
\leq -2\mu \check{\mathcal{E}}+ \frac 1\mu \check{\mathcal N}^2 (\mathcal N + \tilde{\mathcal N} + e^{-\frac 12 |z|}).
\end{equation*}
\end{enumerate}
\end{lemma}
\begin{proof}
(i) The upper bound in~\eqref{coer:ce} is clear.
Next, adapting the proof of (i) of Lemma~\ref{le:ener}, one checks
\begin{equation*}
\check{\mathcal{E}}\ge 2\mu \|\vec{\cce}\|_\ENE^{2}
-\frac{1}{2\mu}\bigg[ \sum_{k,j} |\langle\cce,\partial_{x_j} Q_k \rangle|^{2} +\sum_{k=1,2}\left((\check a_{k}^+)^2+(\check a_{k}^-)^2\right)\bigg].\\
\end{equation*}
The  estimate 
$|\langle \check{\e},\partial_{x_{j}}Q_{k}\rangle|=|\langle\tilde \e, \partial_{x_{j}} \check{Q}_{k}\rangle|
\lesssim \tilde{\mathcal N} \check{\mathcal N}$  from \eqref{c:ortho}
then implies the lower bound in \eqref{coer:ce}.

(ii) We follow the computation of the proof of (ii) of Lemma~\ref{le:ener}. First, 
\begin{align*}
\frac 12\frac \ud{\ud t}\check{\mathcal{E}}&=\int \partial_{t}\cv  \left[-\Delta \cv + (1-\rho\mu)\cv 
-f'(R)\cv\right]+(\partial_{t}\ce +\mu\partial_{t}\cv )\ (\ce+\mu\cv ) \\
&\quad +\int \sum_{k=1,2} (\dot{z}_{k}\cdot\nabla Q_{k})f''(R)\cv^{2}
=\mathbf{g_{1}}+\mathbf{g_{2}}.
\end{align*}
Second, using~\eqref{eq:ce} and integration by parts,
\begin{align*}
\mathbf{g_{1}}&=-\mu\int \left[|\nabla \cv|^{2}+(1-\rho\mu)\cv^{2}-f'(R)\cv^{2}\right]-\rho\int (\ce +\mu\cv)^{2}\\
&\quad+\sum_{k=1,2}\int \cv\left[-\Delta\md_{\cv,{k}}+(1-\rho\mu)\md_{\cv,{k}}-f'(R)\md_{\cv,{k}}\right]\\
&\quad+\sum_{k=1,2}\int (\ce+\mu\cv)\left[\md_{\ce,k}+\mu\md_{\cv,k}\right]
+\int (\ce+\mu\cv) \check{G}.
\end{align*}
The first line in the expression of $\mathbf{g_{1}}$ above is  less than $-\mu\check{ \mathcal E}$
(taking $\mu\leq \rho$).
Next, using~\eqref{eq:z}, \eqref{eq:l}, \eqref{on:checkG}, \eqref{eq:cz} and \eqref{eq:cl},
we check that  the remaining terms in $\mathbf{g_{1}}$ and $\mathbf{g_{2}}$ are estimated
by $\check{\mathcal N}^2(\mathcal N+\tilde{\mathcal N}+e^{-\frac 12 |z|})$.
\end{proof}

Define
\begin{equation}
\check y_{k} = \check z_{k} + \frac{\check\ell_{k}}{2\alpha},\quad
\check{\mathcal F} = \check{\mathcal E} + \sum_{k=1,2} |\check{\ell}_k|^2 + \frac 1{2 \mu} \sum_{k=1,2} (\check{a}_k^-)^2,\quad
\check b =  \sum_{k=1,2} (\check a_{k}^+)^2.
\end{equation}

\begin{lemma}{}\
For $\mu>0$ defined in Lemma~\ref{le:chE}, it holds.
\begin{enumerate}
\item\emph{Comparison.}
\begin{equation}\label{eq:comp}
\mu \check{\mathcal N}^2\leq \mu \|\vec{\check\e}\,\|_{\ENE}+\sum_{k=1,2} |\check{\ell}_k|^2 + \sum_{k=1,2} |\check z_{k}|^2
\leq \check{\mathcal F}+\frac {\check b}{2\mu}+\sum_{k=1,2} |\check z_{k}|^2\lesssim  \check{\mathcal N}^2. 
\end{equation}
\item\emph{Positions.}
\begin{equation}\label{eq:cy}
\bigg| \frac{\ud }{\ud t}{\check y}_{k}\bigg|\lesssim \check{\mathcal N} (\mathcal N+\tilde{\mathcal N}+e^{-\frac 12 |z|})
\lesssim \delta \check{\mathcal N}.
\end{equation}
\item\emph{Damped components.}
\begin{align}
\frac{\ud }{\ud t} \check{\mathcal F}
+2\mu \check{\mathcal F} &\lesssim \check{\mathcal N}^2 (\mathcal N+\tilde{\mathcal N}+e^{-\frac 12 |z|})
\lesssim \delta \check{\mathcal N}^2.\label{on:cF}
\end{align}
\end{enumerate}
\end{lemma}
\begin{proof} (i) follows from~\eqref{coer:ce} (see also the proof of~\eqref{eq:new2} in Lemma \ref{le:new}, (i)).

(ii) Estimate~\eqref{eq:cy} follows  from~\eqref{eq:cz} and~\eqref{eq:cl}.

(iii) is a consequence of Lemmas~\ref{le:chP} and~\ref{le:chE} (see also Lemma \ref{le:new}), (iv)).
\end{proof}

{\bf Step 4.}
For the sake of contradiction, assume
\begin{equation}\label{eq:contra}
\delta^{\frac 12}  \check{\mathcal N}^2(0)<\check b(0).
\end{equation}
We introduce the following bootstrap estimate
\begin{equation}\label{BS:lip}
\delta^{\frac 34}   \check{\mathcal N}^2 <\check b
\end{equation}
and we define
\begin{equation*}
\TS = \sup \left\{t>0\ \text{ such that } \eqref{BS:lip} \text{ holds on } [0,t]\right\}>0.
\end{equation*}
By \eqref{cz-a} and then \eqref{BS:lip}, there holds on $[0,\TS]$
\begin{equation}\label{on:checkb}
\frac{\ud }{\ud t} \check b \geq 2 \nu^+ \check b - C \delta \check{\mathcal N} \sqrt{b} \geq \nu^+ \check b.
\end{equation}
In particular, $\check b$ is positive, increasing on $[0,\TS)$ and 
\begin{equation}\label{checkb:exp} 
\forall t \in [0,\TS), \quad b(t) \ge e^{\nu^+ t} \check b(0).
\end{equation}

By~\eqref{on:cF} and next~\eqref{BS:lip},
\begin{equation*}
\frac{\ud }{\ud t} [ e^{2\mu t}\check{\mathcal F}] 
 \lesssim e^{2\mu t} \delta \check{\mathcal N}^2 \lesssim  e^{2\mu t} \delta^{\frac 14} \check b.
 \end{equation*}
Integrating on $[0,t]$, for any $t\in [0,\TS)$ and using that $\check b$ is increasing, we obtain
\begin{equation*}
\check{\mathcal F}(t)\lesssim  e^{-2\mu t}\check{\mathcal F}(0) +  \delta^{\frac 14}\check b(t),
\end{equation*}
and thus using \eqref{eq:contra},
\begin{equation*}
\delta^{\frac 34} \check{\mathcal F}(t) \lesssim  \delta^{\frac 34}\check{\mathcal F}(0) + \delta\check b(t)
\lesssim \delta^{\frac 14} (\check b(0)+\check b(t))\lesssim \delta^{\frac 14} \check b(t).
\end{equation*}
Using similar argument, we check that
\begin{equation}\label{eq:chell}
\delta^{\frac{3}{4}}\big| \ell_{k}(t)\big|^{2}\lesssim \delta^{\frac{1}{4}}\check{b} (t)\quad \mbox{for}\  k=1,2.
\end{equation}
From~\eqref{eq:cy}, next \eqref{BS:lip} and then~\eqref{on:checkb},
\begin{equation*}
\bigg|\frac{\ud }{\ud t}|{\check y}_{k}|^2\bigg|
\lesssim \delta \check{\mathcal N}^2\lesssim \delta^{\frac 14} \check b
\lesssim \delta^{\frac 14} \frac{\ud }{\ud t} \check b.
\end{equation*}
Integrating on $[0,t]$, for any $t\in [0,\TS)$,
\begin{equation*}
|{\check y}_{k}(t)|^2 \lesssim |{\check y}_{k}(0)|^2 +  \delta^{\frac 14} \check b(t).
\end{equation*}
Thus, using~\eqref{eq:contra} and~\eqref{eq:chell},
\begin{equation*}
\delta^{\frac 34} |{\check z}_{k}(t)|^2\lesssim \delta^{\frac 34} |{\check y}_{k}(t)|^2 + \delta^{\frac 34} |{\check \ell}_{k}(t)|^2
\lesssim \delta^{\frac 14} \check b(t).
\end{equation*}
Using \eqref{eq:comp}, we obtain
\[ \delta^{3/4} \check{\mathcal N}^2 \lesssim \delta^{1/4}  \check b \]
on $[0,\TS)$, and so we have strictly improved estimate~\eqref{BS:lip} for $\delta$ small enough.

By a continuity argument it follows that $\TS=+\infty$. This is contradictory since the exponential growth~\eqref{checkb:exp} would lead to
unbounded $\check b$.

In conclusion, we have just proved that $\check b(0)\leq \delta^{\frac 12}  \check{\mathcal N}^2(0)$. 
Observing 
\begin{equation*}
|a_{k}^+(0)-\tilde a_{k}^+(0)|\lesssim |\check a_{k}^+(0)|+|\langle \vec{\tilde \e}(0),(\vec Z_{k}^+(0)-\vec {\tilde Z}_{k}^+(0)\rangle|
\lesssim \delta^{\frac 14}  \check{\mathcal N}(0),
\end{equation*}
the proof of Proposition~\ref{pr:lip} is complete.
\end{proof}

\subsection{Modulation at initial time}

We introduce some notation related to   initial data as written in the statement of~Theorem~\ref{th:3}
\begin{equation*}
\Omega=(L,\vev), \quad \|\Omega\|=|L|+\|\vev\|_\ENE,\quad h=(h_1,h_2),
\end{equation*}
and as written in Propositions~\ref{pr:exist} and~\ref{pr:lip}
\begin{align*}
& \Gamma=(z_1,z_2,\ell_1,\ell_2,\vvep), \quad \|\Gamma\|=|(z_1,z_2)|+|(\ell_1,\ell_2)|+\|\vvep\|_\ENE,
\quad a^+=(a_{1}^+,a_{2}^+).
\end{align*}
For $\delta>0$, we denote by $\mathcal V_{\delta}^\perp$ the set of $\Omega = (L,\vev)$ satisfying
\begin{equation*}
\left\{\begin{aligned}
&L \in \RR^N \text{ with } |L|> 10|\log \delta|, \\
&\vev\in \mathcal B_\ENE(\delta) \text{ with } \langle \vev,\vec Z^+(\cdot\pm\tfrac L2)\rangle=0, 
\end{aligned}\right.\end{equation*}
and by $\mathcal W_{\delta}^\perp$ the set of $\Gamma=(z_1,z_2,\ell_1,\ell_2,\vvep)$ satisfying
\begin{equation*}\left\{\begin{aligned}
& (z_1,z_2)\in \RR^{2N}   \text{ with }  |z_1-z_2|> 5|\log \delta|,\\
&( \ell_1,\ell_2)\in \mathcal B_{\RR^{2N}}(\delta),\\
& \vvep\in \mathcal B_\ENE(\delta)  \text{ with \eqref{ortho} and $\langle \vvep,\vec Z_k^+\rangle=0$ for $k=1,2$.}
\end{aligned}\right.\end{equation*}

In the statement of Theorem~\ref{th:3}, we do not require the orthogonality conditions $\langle \vev,\vec Z^+(\cdot\pm\tfrac L2)\rangle=0$ as in the definition of the set $\mathcal V_{\delta}^\perp$. Those conditions are obtained by adjusting $(h_{1},h_{2})$ in a second step, see the proof of Theorem~\ref{th:3}. We start with constructing the ``manifold''  for data in $\mathcal V_{\delta}^\perp$, which somehow corresponds to the tangent space and where the discrepancy is superlinear in $\delta$.

\begin{lemma}\label{le:initial}
Assume $\sigma_{1}=1$ and $\sigma_{2}=-1$. There exists $C>0$ such that for all $\delta>0$ small enough, 
and  for any 
$(\Omega,h)\in \mathcal V_{\delta}^\perp\times \mathcal{B}_{\RR^2}(\delta)$,
there exist unique $\Gamma[\Omega,h]  \in \mathcal W_{C\delta}^\perp $ and $a^+[\Omega,h]\in \mathcal{B}_{\RR^2}(C\delta)$ 
such that
\begin{gather}
(\vec Q+h_1 \vec Y^+)(\cdot-\tfrac L2)-(\vec Q+h_2 \vec Y^+)(\cdot +\tfrac L2)+\vev
=\vec Q_1+\vec Q_2+\vec W(a^+)+\vvep,\\
|\beta a^+ - h| \lesssim \delta^{2}. \label{est:a-h}
\end{gather}
Moreover, for any $(\Omega,h)$, $(\tilde \Omega,\tilde h)\in \mathcal V_{\delta}^\perp\times \mathcal{B}_{\RR^2}(\delta)$,
with $|L-\tilde L|<\delta$, it holds
\begin{align}
\|\Gamma[\Omega,h]-\Gamma[\tilde \Omega,\tilde h]\| & \lesssim \|\Omega-\tilde \Omega\|+|h-\tilde h|, \label{est:Gamma_diff} \\
|(\beta a^+[\Omega,h]- h)-(\beta a^+[\tilde \Omega,\tilde h]- \tilde h)| & \lesssim \delta ( \|\Omega-\tilde \Omega\|+|h-\tilde h| ). \label{est:a-h_diff}
\end{align}
\end{lemma}

\begin{proof}
First, (i) of Lemma~\ref{le:dec} implies the existence of $z_1,z_2,\ell_1,\ell_2,\vve$ such that
\begin{equation}\label{pluie}
(\vec Q+h_1 \vec Y^+)(\cdot-\tfrac L2)-(\vec Q+h_2 \vec Y^+)(\cdot +\tfrac L2)+\vev=
 \vec Q_1+\vec Q_2+\vve \end{equation}
where
\begin{equation*}
\|\vve\|_{\ENE}+\sum_{k=1,2}|\ell_{k}|+e^{-2|z_{1}-z_{2}|}\lesssim \delta
\end{equation*}
and  $\vve$ satisfies the orthogonality relations \eqref{ortho}. Using~\eqref{eq:oo1} and projecting~\eqref{pluie} on $\partial_{x_j}Q_{k}$ for $k=1,2$,
we find
\begin{equation*}
|z_1-\tfrac L2|+|z_{2}+\tfrac L2|\lesssim \delta.
\end{equation*}
Hence $|z| \ge L - 2 \delta^{\frac{1}{2}} \ge 5 | \log \delta|$ for $\delta>0$ small enough.

Moreover, \eqref{key1} and \eqref{key2} provide the Lipschitz estimates
\begin{equation}\label{lipest1}
\sum_{k=1,2}\{|z_{k}-\tilde z_{k}|+|\ell_{k}-\tilde\ell_{k}|\}+\|\vve-\vec{\tilde \e}\,\|_{\ENE}\lesssim 
\|\Omega-\tilde \Omega\|+|h-\tilde h|.
\end{equation}
Second, define
\begin{equation*}
a_{k}^+ = \langle \vve,\vec Z_{k}^+\rangle\quad \text{and}\quad \vvep=\vve-\vec W(a_1^+,a_2^+)=\vve-\vec W(a^+),
\end{equation*}
so that $\vvep$ satisfies $\langle \vvep,\vec Z_{k}^+\rangle=0$ for $k=1,2$. Using the definition of $\vvep$ and~\eqref{lipest1}, we obtain 
\begin{align*}
&\sum_{k=1,2}\{|z_{k}-\tilde z_{k}|+|\ell_{k}-\tilde\ell_{k}|\}+\|\vvep-\vec{\tilde \e}_{\perp}\,\|_{\ENE}\\
&\lesssim\sum_{k=1,2}\{|z_{k}-\tilde z_{k}|+|\ell_{k}-\tilde\ell_{k}|\}+\|\vve-\vec{\tilde \e}\,\|_{\ENE}\lesssim 
\|\Omega-\tilde \Omega\|+|h-\tilde h|,
\end{align*}
which is~\eqref{est:Gamma_diff}.
Let us now project \eqref{pluie} on $\vec Z_1^+$. Using the expansion
\[ Q(\cdot - \tfrac L2) - Q_1 = (z_1 - \tfrac L2) \cdot \nabla Q(\cdot - \tfrac L2) + O_{L^2}(\delta^2), \]
the fact that $\langle \partial_{x_j} Q, Y \rangle =0$ for all $j=1, \dots, N$, we get
\[  |\langle \vec Q(\cdot - \tfrac L2) - \vec Q_1, \vec{Z}_1^+ \rangle| \lesssim \delta^2. \]
Similarly,
\[ |\langle \vec Q(\cdot + \tfrac L2) - \vec Q_2, \vec{Z}_1^+ \rangle| = \langle (z_2 +  \tfrac L2) \cdot \nabla Q(\cdot + \tfrac L2), \vec{Z}_1^+ \rangle + O(\delta^2) \lesssim \delta e^{-\frac{1}{2} |z|} +O(\delta^{2})\lesssim \delta^2. \]
Using the assumption $\langle \vev,\vec Z^+(\cdot\pm\tfrac L2)\rangle=0$, we infer
\begin{equation*}
\left|  a^+_{1}- h_{1} \langle \vec Y^+, \vec Z^+ \rangle \right| \lesssim \delta^{2}.
\end{equation*}
One can argue in the same way upon projecting  \eqref{pluie} on $\vec Z_2^+$, and so derive \eqref{est:a-h}.

It remains to prove \eqref{est:a-h_diff}. From \eqref{est:Gamma_diff}, we have
\begin{equation*}
\sum_{k=1,2} |a_{k}^+-\tilde a_{k}^+|+\|\vvep-\vec{\tilde \e}_{\perp}\,\|_{\ENE}\lesssim 
\|\Omega-\tilde \Omega\|+|h-\tilde h|.
\end{equation*}
Using~\eqref{eq:oo1}-\eqref{eq:oo2}, observe as above that the following relation holds:
\begin{align*}
&\big[-\tfrac {\check L}2\cdot (\nabla \vec Q+h_{1}\nabla  \vec Y^+)+\check h_{1} \vec Y^+\big](\cdot-\tfrac{\tilde L}2)
+\big[\tfrac {\check L}2\cdot (\nabla \vec Q+h_{2}\nabla \vec Y^+)-\check h_{2}\vec Y^+\big](\cdot+\tfrac{\tilde L}2)\\
&=\Big(\begin{smallmatrix}\check z_{1}\cdot \nabla  \\ 
- (\check z_{1}\cdot\nabla)(\ell_{1}\cdot \nabla) 
+(\check \ell_{1}\cdot \nabla) \end{smallmatrix}\Big)Q(x-z_{1})
-\Big(\begin{smallmatrix}\check z_{2}\cdot \nabla \\ 
-(\check z_{2}\cdot\nabla)(\ell_{2}\cdot \nabla)
+(\check \ell_{2}\cdot \nabla)\end{smallmatrix}\Big)Q(x-z_{2})\\
&\quad +\vec{W}(\check a^+)-\vec{\check W}(\tilde a^+)
+\vec{\check \e}_{\perp}-\vec{\check\phi}+O_{L^2}(|\check L|^2+|\check h|^2+|(\check z_{1},\check z_{2})|^2).
\end{align*}
Note that we use the notation from \eqref{eq:ch} and
\[ \check W(a) := \sum_{k=1,2} \bigg\{ B_k(a) \check Y_k + \sum_{j=1}^N V_{k,j}(a) \partial_{x_j} \check Q_k \bigg\},\quad \vec{\check{W}}(a):=
\left(\begin{matrix}\check{W}(a)\\
\nu \check{W}(a)\end{matrix}\right).\]
Projecting this relation on $\vec Z_{k}^+$, using the orthogonality relations on $\vvep$, $\vec{\tilde \e}_{\perp}$, $\vec \phi$, $\vec{\tilde \phi}$ and the same argument as in the proof of \eqref{est:a-h}, we find
\[ |\beta( a_{k}^+-\tilde{a}_{k}^{+})-( h_{k}-\tilde{h}_{k})|\lesssim\delta \left(\|\Omega-\tilde \Omega\|+|h-\tilde h|\right)\quad \mbox{for}\ k=1,2,\]
which completes the proof of the lemma.
\end{proof}

\subsection{End of the proof of Theorem~\ref{th:3}}

We first classify $2$-solitary wave whose initial data is chosen so that $\Omega = (L, \vev) \in \mathcal V_{\delta}^\perp$.

\begin{proposition}\label{pr:implicit}
For all $\delta>0$ small enough,
there exists a  map
$H: \mathcal V_{\delta}^\perp \to \bar{\mathcal{B}}_{\RR^2}(\delta)$
such that, given $\Omega = (L, \vev) \in \mathcal V_{\delta}^\perp$ and $h=(h_1,h_2) \in \mathcal B_{\RR^2}(\delta)$, the solution $\vec u$ with initial data
\[ \vec u(0) = (\vec Q+h_1 \vec Y^+)(\cdot-\tfrac L2)-(\vec Q+h_2 \vec Y^+)(\cdot +\tfrac L2)+\vev \]
is a $2$-solitary wave if and only if $h = H(\Omega)$.
Moreover, for any $\Omega$, $\tilde \Omega \in \mathcal V_{\delta}^\perp$ such that $\| \Omega - \tilde \Omega \| < \delta$,
\begin{equation} \label{est:H_lip_Vd}
|H(\Omega)-H(\tilde \Omega)|\lesssim \delta^{\frac 14} \|\Omega-\tilde \Omega\|.
\end{equation}
\end{proposition}

\begin{proof}%[Proof of~Proposition~\ref{pr:implicit}]

First observe that from Propositions~\ref{pr:exist} and~\ref{pr:lip}, for any $\Gamma(0)\in \mathcal W_{\delta}^\perp$, there exists a unique $A^+(\Gamma(0))\in \mathcal B_{\RR^2}(C\delta^{\frac 54})$ such that the solution $\vec v$ of \eqref{nlkg} with initial data 
\begin{equation*}
\vec v(0)=\vec Q_1(0)+\vec Q_2(0)+\vec W(a_1^+(0),a_2^+(0)) +\vvep(0)
\end{equation*}
with $(a_1^+(0),a_2^+(0))\in \mathcal B_{\RR^2}(\delta)$
is a $2$-solitary wave if and only if $(a_1^+(0),a_2^+(0))=A^+(\Gamma(0))$.
Moreover, the following estimate holds
\begin{equation*}
|A^+(\Gamma(0))-A^+(\tilde \Gamma(0))|\lesssim \delta^{\frac 14} \|\Gamma(0)-\tilde \Gamma(0)\|.
\end{equation*}
Using the notation of Lemma \ref{le:initial}, this means that 
\[ \vec u \text{ is a $2$-solitary wave if and only if } a^+[\Omega,h]= A^+(\Gamma[\Omega,h]). \]
 Let us show that this condition can be written $h = H(\Omega)$ for some function $H$.

For $\Omega\in \mathcal V_{\delta}^\perp$ fixed, and $h\in\bar{\mathcal{B}}_{\RR^2}(\delta)$, define
\begin{equation*}
F_{\Omega}(h)=h- \beta \left\{ a^+[\Omega,h] -A^+(\Gamma[\Omega,h])\right\}.
\end{equation*}
First, we observe that $F_{\Omega}:\bar{\mathcal{B}}_{\RR^2}(\delta)\to \bar{\mathcal{B}}_{\RR^2}(\delta)$  is a contraction for $\delta>0$ small enough. Indeed, by Lemma~\ref{le:initial} and Proposition~\ref{pr:exist},
\begin{equation*}
|F_{\Omega}(h)|\leq |h- \beta a^+[\Omega,h]|+|A^+(\Gamma[\Omega,h])|\lesssim \delta^{\frac 54},
\end{equation*}
and
\begin{align*}
|F_{\Omega}(h)-F_{\Omega}(\tilde h)|
& \leq |(h-\tilde h)- \beta (a^+[\Omega,h]-a^+[\Omega,\tilde h])| + \beta |A^+(\Gamma[\Omega,h])-A^+(\Gamma[\Omega,\tilde h])|\\
&\lesssim \delta |h-\tilde h|+ \delta^{\frac 14} \|\Gamma[\Omega,h]- \Gamma[\Omega,\tilde h]\|
\lesssim \delta^{\frac 14} |h-\tilde h|.
\end{align*}
Therefore, by the fixed point theorem, there exists a unique $H= H(\Omega)$ in $\bar{\mathcal{B}}_{\RR^2}(\delta)$ such that
$F_{\Omega}(H)=H$, which is equivalent to 
$a^+[\Omega,H]=A^+(\Gamma[\Omega,H])$.

Second, for any $\Omega$, $\tilde \Omega \in \mathcal V_{\delta}^\perp$,
\begin{align*}
 |H(\Omega)-H(\tilde \Omega)| &=
|F_{\Omega}(H(\Omega))-F_{\tilde \Omega}(H(\tilde \Omega))|\\
&  \leq |F_{\Omega}(H(\Omega))-F_{\Omega}(H(\tilde \Omega))|+|F_{\Omega}(H(\tilde \Omega))-F_{\tilde \Omega}(H(\tilde \Omega))|\\
&  \leq \frac 12 |H(\Omega)-H(\tilde \Omega)|+|F_{\Omega}(H(\tilde \Omega))-F_{\tilde \Omega}(H(\tilde \Omega))|
\end{align*}
and so
\[ |H(\Omega)-H(\tilde \Omega)|\leq
2|F_{\Omega}(H(\tilde \Omega))-F_{\tilde \Omega}(H(\tilde \Omega))|. \]
By definition of $F_{\Omega}(h)$, Lemma~\ref{le:initial} and Proposition~\ref{pr:exist}, one has
\begin{align*}
|F_{\Omega}(h)-F_{\tilde \Omega}(h)|
&\leq \beta \left( | a^+[\Omega,h]-a^+[\tilde\Omega,h]|+|A^+(\Gamma[\Omega,h])-A^+(\Gamma[\tilde\Omega,h])| \right) \\
&\lesssim\delta\|\Omega-\tilde \Omega\|+ \delta^{\frac 14} \|\Gamma[\Omega,h]- \Gamma[\tilde \Omega, h]\|
\lesssim\delta^{\frac 14} \|\Omega-\tilde \Omega\|,
\end{align*}
which yields
$|H(\Omega)-H(\tilde \Omega)|\lesssim \delta^{\frac 14} \|\Omega-\tilde \Omega\|$.
\end{proof}

\begin{proof}[Proof of Theorem \ref{th:3}]
The map $H$ was constructed in Proposition \ref{pr:implicit} (locally) on a subspace. Our goal is now to extend it to the full open set given in the statement of the Theorem. Let $\delta>0$ to be fixed later. Given $(L,\vec \phi)$ such that $|L| \ge 10 |\log \delta|$ and $\| \vec \phi \|_{H^1 \times L^2} < \delta$, we decompose
\[ \vec \phi = - h_{1,\parallel} \vec Y^+( \cdot - \tfrac L2) +  h_{2,\parallel} \vec Y^+(\cdot + \tfrac L2) + \vec{\phi}_\perp, \quad \langle \vec \phi_\perp, \vec Z^+( \cdot \pm \tfrac L2) \rangle =0. \]
These conditions are a linear system on $ h_{k,\parallel} =  h_{k,\parallel}(L,\vec{\phi})$ of the form
\[ \begin{cases}
 h_{1,\parallel} = - \beta \langle \vec \phi,  \vec Z^+( \cdot - \tfrac L2) \rangle + O(e^{-L/2}  h_{2,\parallel}) \\
 h_{2,\parallel} = \beta \langle \vec \phi,  \vec Z^+( \cdot +\tfrac L2) \rangle + O(e^{-L/2}  h_{1,\parallel})
\end{cases} \]
which can be inverted for $\delta>0$ small enough, and furthermore, for any such $\delta$, one has
\begin{equation} \label{est:proj_Vd}
| h_{1,\parallel}| + | h_{2,\parallel}| + \| \vec{\phi}_{\perp} \|_{H^1 \times L^2} \le C \| \vec{\phi} \|_{H^1 \times L^2},
\end{equation}
and the Lipschitz estimates
\begin{equation}\label{est:diff:proj_Vd}
\sum_{k=1,2}| h_{k,\parallel}-\tilde{h}_{k,\parallel}| + \| \vec{\phi}_{\perp} -\vec{\tilde{\phi}}_{\perp} \|_{H^1 \times L^2} \le C\left(\big|L-\tilde{L}\big| +\| \vec{\phi}-\vec{\tilde{\phi}} \|_{H^1 \times L^2}\right),
\end{equation}
for some constant $C$ independent of small $\delta >0$. In particular, $\vec{\phi}_\perp \in \mathcal V_{C \delta}^\perp$ and up to lowering $\delta$, we can assume that Proposition \ref{pr:implicit} applies on that set. Observe that our initial data can be written
\[ \vec u(0) = \left( \vec{Q} + (h_1 -  h_{1,\parallel})  \vec Y^+ \right) \left( \cdot - \tfrac L2 \right) - \left( \vec{Q} + (h_2 -  h_{2,\parallel})  \vec Y^+ \right) \left( \cdot + \tfrac L2 \right) + \vec \phi_\perp. \]
Proposition \ref{pr:implicit} asserts that $\vec u$ is a $2$-solitary wave if and only if $H(L, \vec \phi_\perp) = (h_1 -  h_{1,\parallel}, h_2 -  h_{2,\parallel})$ or equivalently that $(h_{1,\parallel},h_{2,\parallel}) + H(L, \vec \phi_\perp) = (h_1,h_2)$.

We are therefore led to define the extension of $H$ as
\[ H(L, \vec \phi) :=  (h_{1,\parallel},h_{2,\parallel}) + H(L,\vec{\phi}_\perp). \]
(where $H(L,\vec{\phi}_\perp)$ is given in Proposition \ref{pr:implicit}). Due to~\eqref{est:H_lip_Vd},~\eqref{est:proj_Vd} and~\eqref{est:diff:proj_Vd},  $H(L,\vec{\phi})$ is a Lipschitz map. In conclusion, $H$ meets the requirements of Theorem \ref{th:3}.
\end{proof}

\end{document}